\documentclass[10pt,leqno]{article}  
\usepackage[english]{babel}
\usepackage{graphicx}
\usepackage[centertags]{amsmath}
\usepackage{amsfonts,amssymb,amsthm}

\numberwithin{equation}{section}

\theoremstyle{plain}
\newtheorem{theorem}{Theorem}
\newtheorem{proposition}{Proposition}
\newtheorem{lemma}{Lemma}

\theoremstyle{definition}
\newtheorem{remarkbase}{Remark}
\newenvironment{remark}{\pushQED{\qed}\remarkbase}{\popQED\endremarkbase} 

\newtheorem{observation}{Observation}

\newtheoremstyle{hypstyle}{}{}{}{}{\bfseries}{.}{ }%
{\thmname{#1}\thmnumber{ (H#2)}\thmnote{}}

\theoremstyle{hypstyle}
\newtheorem{hypbase}{Hypothesis} 
\newenvironment{hyp}{\pushQED{\qed}\hypbase}{\popQED\endhypbase}

\newcommand{\N}{{\mathbb N}}
\newcommand{\R}{{\mathbb R}}
\newcommand{\C}{{\mathbb C}}

\newcommand{\CC}{\mathbb C}

\newcommand{\mA}{\mathcal{A}}
\newcommand{\mC}{\mathcal{C}}
\newcommand{\mD}{\mathcal{D}}
\newcommand{\mE}{\mathcal{E}}
\newcommand{\mF}{\mathcal{F}}
\newcommand{\mH}{\mathcal{H}}
\newcommand{\mL}{\mathcal{L}}
\newcommand{\mW}{\mathcal{W}}
\newcommand{\mS}{\mathcal{S}}

\renewcommand{\a}{\alpha}
\renewcommand{\b}{\beta}
\newcommand{\g}{\gamma}
\renewcommand{\d}{\delta}
\newcommand{\D}{\Delta}
\newcommand{\e}{\varepsilon}
\newcommand{\ph}{\varphi}
\newcommand{\lm}{\lambda}
\newcommand{\m}{\mu}
\newcommand{\Om}{\Omega}
\newcommand{\p}{\pi}
\newcommand{\s}{\sigma}
\renewcommand{\t}{\tau}
\renewcommand{\th}{\vartheta}
\newcommand{\Th}{\Theta}

\renewcommand{\Re}{\mathrm{Re}\,}
\renewcommand{\Im}{\mathrm{Im}\,}

\newcommand{\gr}{\nabla}

\newcommand{\intp}{\int_{0}^{2\p}} 

\newcommand{\vp}{\varpi}
\newcommand{\br}{\boldsymbol{r}}
\newcommand{\tchi}{\widetilde{\chi}}

\title{Steady periodic water waves under nonlinear elastic membranes}

\author{Pietro Baldi and John F. Toland}

\date{}

\begin{document}

\maketitle

\begin{abstract}
\noindent  This is a study of  two-dimensional steady periodic travelling waves on the surface of an infinitely deep irrotational ocean, when the top streamline is in contact with a membrane which has a nonlinear response to stretching and bending, and the pressure in the air above is constant. It is not supposed that the waves have small amplitude. The problem of existence of such waves is addressed using methods from the calculus of variations. The analysis involves 
the Hilbert transform and a Riemann-Hilbert formulation.
\end{abstract}

\begin{small}

\end{small}


{\footnotesize\tableofcontents}

\section{Introduction}

Regarding water as an inviscid incompressible liquid, we study two-dimensional steady waves on the surface of an
 ocean of infinite depth,  moving  under the influence of gravity when the surface is in contact with a thin frictionless elastic membrane that responds nonlinearly to bending, compression and stretching, and above the membrane there is constant atmospheric pressure. We suppose that the steady fluid motion is irrotational and the top streamline is a space-periodic curve that travels with constant velocity, without changing its shape. 
We suppose also that the two-dimensional cross-section of the elastic surface behaves mechanically like a thin (unshearable) hyperelastic Cosserat rod, as  described by Antman in \cite{Antman}, Ch.\, 4. 
The physical significance of such a problem is evident; 
for example in the theory of very large floating structures  or platforms  (see \cite{Andrianov-PhD} and the references therein), or, possibly, flow under ice. 
We refer to this as a hydroelastic travelling wave problem.

\medskip

The mathematical study of these waves began with the linear theory of Greenhill in the nineteenth century \cite{Greenhill}, but
an analysis of nonlinear models has only recently been attempted.
 In \cite{Toland-steady},  the existence question was formulated as a variational problem, and existence was proved, for the case of
a class of membranes that have an infinite elastic energy when the stretching or bending exceed certain fixed values, by maximising a Lagrangian over a set of admissible functions.
 Other recent work extend the theory of \cite{Toland-steady} in different ways. For example,  in \cite{Toland-heavy}  membranes with positive densities are  included in the theory provided the resulting variational problem is convex. This is a restriction on the membrane density and on one of the wave-speed parameters in the problem. In \cite{Plotnikov-Toland-nonconvex}, that restriction was removed and  the general problem of membranes with positive mass was studied, using Young's measures to deal with the problem of non-convexity.

\medskip

In the present paper we  generalize and simplify the theory of surface membranes with zero density \cite{Toland-steady}, by proving the existence of steady periodic hydroelastic waves for membranes when the stored elastic energy remains finite but has power-law growth as the bending or stretching/compression  increases, as is  more or less standard in the
mathematical theory of nonlinear elasticity.
A further novelty  is the use of a Riemann-Hilbert formulation in the context of hydroelastic waves. This approach simplifies and clarifies the reduction of the problem to one for a single function of a single real variable.

\medskip

In the rest of this Introduction, we describe the physical problem,  summarize the main results and the methods, and discuss the hypotheses on the elastic properties of the membranes under which they are obtained.

\subsection{The physical problem} \label{physical}

The physical system under investigation was studied in \cite{Toland-steady}, Sec.\,1.1.
We seek waves that are two-dimensional and steady with prescribed period.
More precisely, we consider waves such that
\begin{itemize}
\item
[($i$)] in  three-dimensional space $(X,Y,Z)$, with gravity $g$ acting in the negative $Y$ direction, the flow beneath the free surface is irrotational;
\item
[($ii$)] the $Z$-component of the fluid velocity is everywhere zero and all components tend to zero as $Y \to -\infty$;
\item
[($iii$)] the $Y$-coordinates of points on the surface are independent of $Z$ and the surface moves without change of form and with constant speed $c$ in the $X$-direction;
\item
[($iv$)] the flow is $2\p$-periodic and stationary with respect to axes moving with the wave speed.
\end{itemize}
Because the membrane has zero density, it is equivalent to study, in a frame moving with the wave, steady $2\p$-periodic waves for which the speed of the flow at infinite depth is $-c$  horizontally.
In this frame, the intersection of the surface membrane with the plane $Z=0$, called the membrane section, is supposed to behave like a nonlinear, unshearable,  hyperelastic rod for which the stored energy depends on stretch and curvature.
By the reference membrane is meant the line $Y=0$ and one period of it refers to a line segment of length $2\p$.
We study waves for which
\begin{itemize}
\item
[($v$)] one period of the reference membrane is deformed to become one period of the hydroelastic wave surface.
\end{itemize}

The unknown region occupied by the liquid is characterized by the kinematic requirement that the surface is a streamline and the dynamic condition that the pressure $P$ in the fluid and internal forces  are those required to deform the membrane. Therefore a steady hydroelastic wave with speed $c$ satisfying ($i$-$iv$) corresponds to a non-self-intersecting smooth curve $\mS$ in the plane $(X,Y)$ which is $2\p$-periodic in the horizontal direction $X$ and for which there exists a solution of the following system:
\begin{subequations}  \label{classical problem}
\begin{align}
\label{D psi=0}
\D \psi  = 0 & \quad \text{below } \mS,  \\
\label{psi=0}
\psi = 0 & \quad \text{on } \mS
\  \text{(the kinematic boundary condition),}  \\
\label{psi bottom}
\gr \psi (X,Y) \to (0,c) & \quad \text{as } Y \to - \infty,
\end{align}
with the dynamic boundary condition
\begin{equation}  \label{psi dynamics}
\frac12\, |\gr \psi|^2 + g Y = \frac{c^2}{2} - P
\quad \text{on } \mS.
\end{equation}
Moreover, suppose that $\br$ is the physical deformation that carries a material point $x$ of the reference membrane into its new position $\br(x)$. Then, by assumption, the profile $\{\br(x) : x \in \R \}$ of the deformed membrane must coincide with the free surface $\mS$ of the fluid, and the constraint ($v$) reads
\begin{equation}  \label{r constrain}
\mS \cap \{ 0 \leq X \leq 2\p \}
= \{ \br(x) : x \in [x_0, x_0 + 2\p] \}
\end{equation}
\end{subequations}
for some $x_0 \in \R$.
In \cite{Toland-steady}, Antman's treatment \cite{Antman} of unshearable Cosserat rods
 is used to derive a formula for the pressure $P$ in  \eqref{psi dynamics}.
Here we simply recall that formula after introducing some notation.

\medskip

Consider an interval of membrane in its rest position. Its material points are labelled $x \in [x_1, x_2]$. We consider a deformation $\br$ that move any point $x$ in its new position $\br(x)$.
The stretch of the deformed membrane at the point $\br(x)$ is then
\begin{subequations}  \label{def nu mu}
\begin{equation}  \label{def nu}
\nu(x) := |\br'(x)|,
\end{equation}
where $'$ denotes the derivative with respect to $x$.
Let $\th(x)$ denote the angle formed by the membrane and the positive horizontal semiaxis at the point $\br(x)$, and let
\begin{equation}  \label{def mu}
\mu(x) := \th'(x).
\end{equation}
\end{subequations}
Then the curvature of the membrane at $\br(x)$ is
\[
\hat\s(\br(x)) := \frac{\mu(x)}{\nu(x)}\,.
\]
We  assume  that the material is hyperelastic 
(see, for example, \cite{Antman}, Ch.\,4)
with stored elastic energy function,
\[
E(\nu,\mu) \geq 0,
\quad
\nu > 0, \ \mu \in \R,
\]
of class $C^2$.
Denote by $E_1$ and $E_2$ the partial derivatives of $E$ with respect to its variables, $\nu$ and $\mu$, respectively.
From the balance law for forces and moments acting on the membrane, it follows (see \cite{Toland-steady}, eqs.~(1.1)) that
\begin{subequations}
\begin{gather}  \label{palp}
\nu(x) \, E_1(\nu(x),\mu(x))' + \mu(x) \, E_2(\nu(x),\mu(x))' = 0,\\\intertext{and}
  \label{def P}
P(\br(x)) = \frac{1}{\nu(x)} \,
\Big( \frac{E_2(\nu(x), \mu(x))'}{\nu(x)} \Big)'
- \hat\s(\br(x)) \, E_1(\nu(x), \mu(x)),
\end{gather}
\end{subequations}
where $P(\br)$ is the pressure that is needed to produce the deformation $\br$.
Hence the physical deformation $\br$ of the material points of the reference membrane enters in the hydroelastic wave problem  \eqref{psi dynamics}, through the term $P$.

\begin{remark} At this point in \cite{Toland-steady}, eqn.~(1.6), $\nu(x)$ was  calculated in terms of $\hat\s(\br(x))$, using equation \eqref{palp} and  the constraint $(v)$.
 In the present paper we  avoid that calculation  at this stage. The formula for $\nu$  will emerge later, in the final part of the regularity proof, see \eqref{def vp}.
\end{remark}

Now we introduce the hypotheses on $E$ and explain briefly their roles in the theory.

\subsection{Hypotheses}

 The first seven hypotheses are used in the theory of maximization of the Lagrangian. The next five are needed to
 ensure that maximizers satisfy the Euler-Lagrange equation and are sufficiently regular to give a solution of the physical problem
\eqref{classical problem}.
The first hypothesis  is needed to define the potential energy stored by elasticity in the deformed membrane. The absence of a shear variable (see \cite{Antman}, Ch.\,4) in the argument of $E$  reflects our assumption that the membrane is thin and unshearable.
\begin{hyp}
\label{hyp:exists E}(\emph{Unshearable hyperelasticity})
There exists a stored elastic energy function,
\[
E(\nu,\mu) \geq 0,
\quad
\nu > 0, \ \mu \in \R,
\]
of class $C^2$,
such that the elastic energy in a segment  $[x_1, x_2]$ of material, when deformed by $x \mapsto \br (x)$, is
\begin{equation}  \label{def elastic}
\mE (\br) = \int_{x_1}^{x_2} E(\nu(x), \mu(x)) \, dx,
\end{equation}
where $\nu(x), \mu(x)$ are defined in \eqref{def nu mu}.
\end{hyp}

The following  four hypotheses are used to obtain the existence of a maximizer of the Lagrangian $J_0$, using the direct method of the calculus of variations. First we assume that the material response is even with respect to curvature and then that the elastic energy  is minimized when the material is neither stretched nor bent (we normalize the elastic energy of this rest state to be zero).
\begin{hyp}  \label{hyp:E even mu}
(\emph{Evenness with respect to curvature})\,
$E(\nu,\mu) = E(\nu, - \mu)$\, for all $\nu > 0$, $\mu \in \R$.
\end{hyp}

\begin{hyp}  \label{hyp:E=0}(\emph{Rest state})
$E(\nu,\mu) \geq 0 = E(1,0)$\, for all $\nu>0$, $\mu \in \R$.
\end{hyp}
The next condition ensures the upper semi-continuity of the functional to be maximized. 
However, it has a wider significance.
For example, in the dynamic theory of nonlinear rods
it ensures that the equations of motion  are strictly hyperbolic
and accordingly have rich wave-like behavior. 
It is also an exact analog of the 3-dimensional Strong Ellipticity Condition \cite{morrey}. 
By Lemma \ref{lemma:E* convex} below, (H\ref{hyp:E jointly convex}) coincides with the convexity assumption in \cite{Toland-heavy}.
\begin{hyp}  
\label{hyp:E jointly convex} (\emph{Strict joint convexity})
\[
E_{22} > 0,
\quad
E_{11} > 0,
\quad
E_{11} \, E_{22} - (E_{12})^2 > 0
\]
at all points $(\nu,\mu)$ of the semi-plane $\nu>0$, $\mu \in \R$.
\end{hyp}

It is natural to assume that  for infinite stretch, compression, or  curvature, an infinite amount of energy is required, that is,
$E(\nu,\mu)$ tends to infinity as $\nu$ goes to 0 or $+\infty$, or $|\mu|$ goes to $+\infty$.
The next hypothesis quantifies that assumption.

\begin{hyp}  
\label{hyp:basic growth} (\emph{Growth condition})
\[
E(\nu,\mu) \geq K_0 \Big( \nu^r + \frac{1}{\nu^s}\, + |\mu|^p \Big) - K_0'
\quad
\forall \nu > 0, \ \mu \in \R,
\]
for some positive constants $K_0, K_0'$,
and some exponents $r>2$, $s>0$, $p>1$.
\end{hyp}
With the existence question settled under the above hypotheses we need two further hypotheses to guarantee basic properties of the maximizer. The first ensures that the maximizer is non-trivial (does not correspond to the laminar flow of a wave with zero elevation), see Lemma \ref{lemma:Sigma>0}.

\begin{hyp}  \label{hyp:Sigma}(\emph{Non-trivial maximizers})
$c^2 > g + E_{22}(1,0)$.
\end{hyp}
The second is needed to show that the curve that emerges is non-self-intersecting
(see Lemma \ref{lemma:compactness}).
Obviously, this is essential if it is to be the surface of a travelling wave. However (H\ref{hyp:special}) has a further role. It guarantees sufficient compactness  of a maximizing sequence  to yield the existence of a maximizer.
For $\ell > 1$, let $A(\ell)$ be the area in a circle enclosed between an arc of length $2\p \ell$ and a chord of length $2\p$ (Figure 1).
The asymptotics for $A(\ell)$ is then
\begin{equation*}
\lim_{\ell \to 1} \frac{A(\ell)}{\sqrt{\ell-1}} \to 2\sqrt{2/3}\,\p^2,
\qquad
\lim_{\ell \to +\infty} \frac{A(\ell)}{\ell^2} = \pi,
\end{equation*}
see Lemma \ref{lemma:A}.

\begin{figure}[ht]   \label{fig1}
\includegraphics{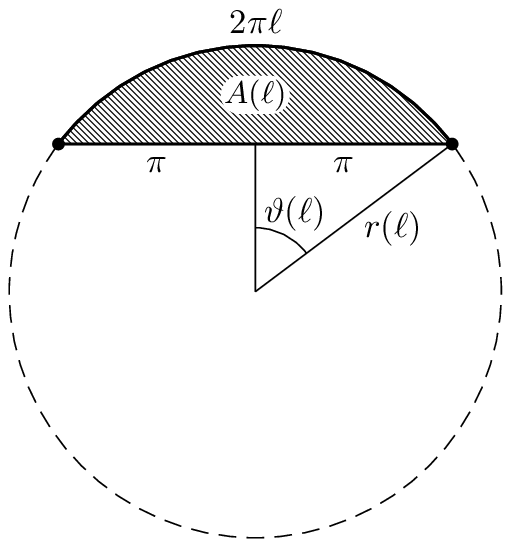}
\hspace{6mm} 
\includegraphics
{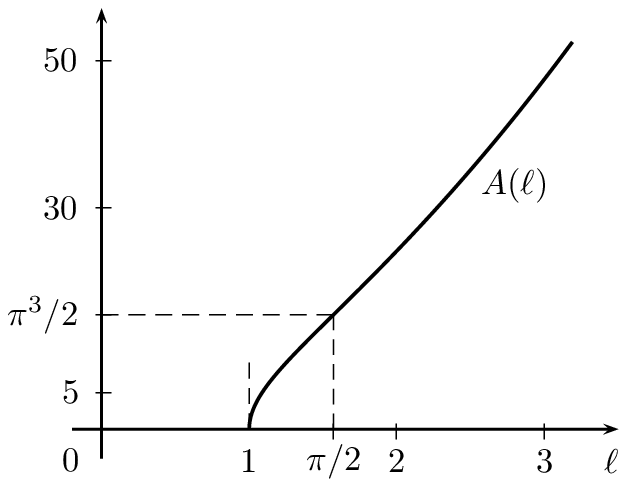}
\caption {Geometric definition and plot of $A(\ell)$}
\end{figure}

\begin{hyp}
\label{hyp:special} (\emph{Non-self-intersecting maximizers})
There exists $\mu^* \in (0,1)$ such that
\[
E(\nu,\mu^*) \geq
\frac{g}{8\p^2 } A(\nu)^2
+\frac{c^2}{4\p}A(\nu)
+\frac g2 \, A(\nu) \sqrt{\nu^2-1}
\quad \forall \nu \geq 1.
\qedhere 
\]
\end{hyp}

\begin{remark}  \label{rem:brutal special}
Note that, by \eqref{A control} below, (H\ref{hyp:special}) is implied by the stronger condition
\[
E(\nu,\mu^*) \geq \frac{g}{2}\, \nu^4 + g \p \, \nu^3 + \frac{c^2}{2}\, \nu^2
\quad \forall \nu \geq 1,
\]
and the exponent $r$ in (H\ref{hyp:basic growth}) must be not less than 4.
\end{remark}

\begin{remark}
(H\ref{hyp:Sigma}) and (H\ref{hyp:special}) 
are simultaneously satisfied if
\begin{equation}  \label{compatibility}
g + E_{22}(1,0) \,<\, c^2 \, \leq \,
\inf_{\nu > 1}
\Big\{ \frac{4\p E(\nu,\mu^*)}{A(\nu)}\, - \frac{g A(\nu)}{2\p}\,
- 2\p g \sqrt{\nu^2-1}\,\Big\}.
\end{equation}
As an example, we consider the  case when $E$ splits, $E(\nu,\mu) = S(\nu) + B(\mu)$, and show that the interval \eqref{compatibility} for $c^2$ is nonempty provided the growth of $S(\nu)$ for large $\nu$ and that of $B(\mu)$ on the interval $\mu \in [0,\mu^*]$ are sufficiently rapid.

We denote $b := B''(0)$. Let us fix $M>0$. If the growth of $S(\nu)$ as $\nu \to +\infty$ is rapid enough, then there exists $\bar\nu > 1$ (depending on $M$ and $b$) such that
\[
\frac{4\p S(\nu)}{A(\nu)}\, - \frac{g A(\nu)}{2\p}\, - 2\p g \sqrt{\nu^2-1}\,
\geq \, g + b + M   \quad \forall \nu \geq \bar{\nu}.
\]
Next, since $A(\nu)$ is increasing,
\[
\frac{4\p B(\mu^*)}{A(\nu)}\, - \frac{g A(\nu)}{2\p}\, - 2\p g \sqrt{\nu^2-1}\,
\geq
\frac{4\p B(\mu^*)}{A(\bar{\nu})}\, - \frac{g A(\bar{\nu})}{2\p}\, - 2\p g \sqrt{ \bar{\nu}^2 - 1 }
\]
for all $\nu \in (1,\bar{\nu}]$. Then
\[
\frac{4\p E(\nu,\mu^*)}{A(\nu)}\, - \frac{g A(\nu)}{2\p}\, - 2\p g \sqrt{\nu^2-1}\,
\geq \, g + b + M
\quad \forall \nu > 1
\]
provided
\begin{equation}  \label{provided example}
B(\mu^*) \geq \frac{g}{8\p^2}\, A(\bar{\nu})^2 + \frac{g}{2}\, A(\bar{\nu}) \sqrt{ \bar{\nu}^2-1 } \,
+ \frac{(g + b + M)}{4\p}\, A(\bar{\nu}).
\end{equation}
\eqref{provided example} holds if $B(\mu^*)$ is sufficiently large, depending on $b$ and $M$.
For example, if $B(\mu)=(b/2) \mu^2 + b_1 \mu^4$, then \eqref{provided example} holds if $b_1$ is sufficiently large.
In that case, \eqref{compatibility} holds for all $c^2$ in the interval
\[
g + b \, \leq \, c^2 \, \leq \, g + b + M. 
\qedhere 
\]
\end{remark}

The remaining five hypotheses are needed to show that  maximizers
of the Lagrangian yield steady hydroelastic travelling waves. The first is an assumption on  $\nabla E$, which will lead to the conclusion that the stretch $\nu$ of the membrane is bounded above (see Lemma \ref{lemma:chi bound below}).
\begin{hyp}  
\label{hyp:uniform infty}(\emph{Bounded stretch})
\[
\lim_{\nu \to +\infty}
\big\{ \inf_{\mu \in \R} \big( \gr E(\nu,\mu) \cdot (\nu,\mu) - E(\nu,\mu) \big) \big\} = +\infty.
\qedhere 
\]
\end{hyp}

\begin{remark}  \label{rem:S+B}
In the ``splitting'' case when $E(\nu,\mu) = S(\nu)+B(\mu)$,
(H\ref{hyp:uniform infty}) is automatically satisfied when 
(H\ref{hyp:E jointly convex},\ref{hyp:basic growth}) hold. 
Indeed,
\[
\big(B'(\mu)\,\mu - B(\mu) \big)' = B''(\mu)\,\mu
\]
has the same sign as $\mu$, therefore
\[
\inf_{\mu \in \R} \big( \gr E(\nu,\mu) \cdot (\nu,\mu) - E(\nu,\mu) \big) =
S'(\nu)\,\nu - S(\nu) - B(0).
\]
Now, $S'(\nu)\,\nu - S(\nu)$ is strictly increasing in $\nu$, thus its limit
as $\nu \to +\infty$ exists.
Suppose that such a limit is a real number. Then
\[
\Big( \frac{S(\nu)}{\nu}\,\Big)' = \frac{S'(\nu)\,\nu - S(\nu)}{\nu^2}\, \to 0
\quad \text{as } \nu \to +\infty.
\]
Hence there exists $\bar\nu$ such that
\[
\Big| \Big( \frac{S(\nu)}{\nu}\,\Big)' \Big| \leq 1
\quad \forall \nu \geq \bar\nu.
\]
It follows that
\[
\frac{S(\nu)}{\nu}\, =
\frac{S(\bar\nu)}{\bar\nu}\, + \int_{\bar\nu}^\nu
\Big( \frac{S(\xi)}{\xi}\,\Big)'\,d\xi
\, \leq \,
C + \nu
\ \quad \forall \nu \geq \bar\nu,
\]
for some constant $C$.
But this violates (H\ref{hyp:basic growth}), because $r>2$.
Hence $S'(\nu)\,\nu - S(\nu)$ goes to $+\infty$ as $\nu \to +\infty$, and
(H\ref{hyp:uniform infty}) follows.
\end{remark}

The next two growth conditions ensure the differentiability of $J_0$ at a maximizer.
\begin{hyp}  
\label{hyp:E1 growth control}
There exist positive constants $K_1, \bar\nu_1$ and $\bar\mu_1$ such that
\[
|E_1(\nu,\mu)| \leq K_1 \Big( \frac{1}{\nu^{s+1}}\, + |\mu|^p \Big)
\quad \forall \nu \leq \bar\nu_1, \ |\mu| \geq \bar\mu_1.
\qedhere
\]
\end{hyp}

\begin{hyp} \label{hyp:E2 growth control}
There exist positive constants $K_2, \bar\nu_2$ and $\bar\mu_2$ such that
\[
|E_2(\nu,\mu)| \leq
K_2 \Big( \frac{1}{\ \nu^{\frac{s(p-1)}{p}}\ } \, + |\mu|^{p-1} \Big)
\quad \forall \nu \leq \bar\nu_2, \ |\mu| \geq \bar\mu_2.
\qedhere
\]
\end{hyp}

\begin{remark} \label{remark:Young}
Recalling Young's inequality
\begin{equation}  \label{Young}
xy \leq \e \, x^q + C_\e \, y^{q'} \quad \forall \, x,y,\e > 0,
\end{equation}  
with $q>1$, $1/q + 1/q' = 1$ and $C_\e := \e^{-1/(q-1)}$, 
we note that (H\ref{hyp:E1 growth control},\ref{hyp:E2 growth control}) are compatible with the presence of ``mixed term'' 
of the type $|\mu|^\a / \nu^\d$ in $E(\nu,\mu)$, 
provided these couplings are not too strong with respect to the leading ``pure terms'' of the form $\nu^r$, $1/\nu^s$ and $|\mu|^p$;
see the example in Subsection \ref{subsec:example}.
\end{remark}

\begin{hyp}  \label{hyp:new magic}
For every $\g \in \R$ there exist positive constants $K_\g$, $K'_\g$, $\bar\nu_\g$, $\bar\mu_\g$ with the following property.
If $(\nu,\mu)$, with $\nu \leq \bar\nu_\g$ or $|\mu|\geq \bar\mu_\g$,
satisfy
\[
E(\nu,\mu) - \gr E(\nu,\mu) \cdot (\nu,\mu) = \g,
\]
then
\[
\frac{K_\g}{\nu^s}\,
\leq  |\mu|^p
\leq \frac{K'_\g}{\nu^s} \,.  
\qedhere
\]
\end{hyp}

\begin{remark}  \label{remark:new magic splitting case}
In the case when $E(\nu,\mu) = S(\nu) + B(\mu)$,
one can show that, if (H\ref{hyp:E jointly convex},\ref{hyp:basic growth},\ref{hyp:E1 growth control},\ref{hyp:E2 growth control}) hold, 
then
\[
-\,\frac{C}{\nu^s}\, \leq \nu S'(\nu) - S(\nu)
\leq -\,\frac{C'}{\nu^s}\,,
\quad
\mu B'(\mu) - B(\mu) \leq C'' |\mu|^p
\]
for all small $\nu$, all large $|\mu|$, for some positive constants $C,C',C''$.
If, in addition,
\begin{equation}  \label{B more than convex}
\mu B'(\mu) - B(\mu) \geq C |\mu|^p
\end{equation}
for all large $|\mu|$, for some $C$, then (H\ref{hyp:new magic}) is satisfied.
We note that \eqref{B more than convex} holds if and only if the ratio $B(\mu)/\mu^\a$ is non-decreasing for all large $\mu$, for some $\a>1$.
Also, if $B$ satisfies
\[
C_0 |\mu|^{p-1} \leq |B'(\mu)| \leq C_1 |\mu|^{p-1}
\]
for all $\mu$ large, for some $C_0, C_1 > 0$ such that $C_1 < p \, C_0$,
then $B(\mu) \leq C + (C_1 / p) |\mu|^p$ for all $\mu$ large, therefore \eqref{B more than convex} holds.

Finally, we note that the simplest case $B(\mu) = |\mu|^p + $ (lower order terms) satisfies \eqref{B more than convex} trivially.
\end{remark}
The final assumption leads to regularity properties of solutions (see Lemma \ref{lemma:bootstrap}).

\begin{hyp}  \label{hyp:E2 >}
There exist positive constants $K_3,\bar\nu_3, \bar\mu_3$ and an exponent $\a := s(p-1)/p - \e$, with $\e>0$, such that
\[
|E_2(\nu,\mu)| \geq K_3 \nu^\a |\mu|^{p-1}
\quad \forall \nu \leq \bar\nu_3, \ |\mu| \geq \bar\mu_3.
\qedhere
\]
\end{hyp}

\begin{remark}
In the splitting case $E(\nu,\mu) = S(\nu) + B(\mu)$, 
(H\ref{hyp:E2 >}) is automatically satisfied when 
(H\ref{hyp:E jointly convex},\ref{hyp:basic growth}) 
hold, because, by the convexity of $B(\mu)$ and its growth condition,
\[
|E_2(\nu,\mu)| = |B'(\mu)| \geq \frac{B(\mu)}{|\mu|}\, \geq C |\mu|^{p-1}
\]
for all $|\mu|$ sufficiently large, uniformly in $\nu$.
\end{remark}

\subsection{Main result and methods}
Under the above hypotheses on the elastic properties of the membrane,  our main result on the existence of   $\mS, \psi$ and $\br$ satisfying the hydroelastic wave problem \eqref{classical problem} is the following:

\begin{theorem}  \label{thm:physical}
\emph{(Existence).}
Suppose that the stored elastic energy function $E(\nu,\mu)$ satisfies \emph{(H\ref{hyp:exists E}-\ref{hyp:E2 >})}.
Then, for admissible velocities $c^2$ in a certain interval, see  \eqref{compatibility},
there exist a free surface curve $\mS$ of class $W^{3,\infty}$
and a membrane  deformation $\br$ of class $W^{2,\infty}$,
satisfying the constraint \eqref{r constrain},
such that the stream function $\psi$ that solves \emph{(\ref{classical problem}a,b,c)}
is also a solution of the dynamic boundary equation \eqref{psi dynamics}.

\emph{(Regularity).}
If $E \in C^k$, with $k \geq 2$, then $\mS$ is of class $W^{k+1,\infty}$, and $\br \in W^{k,\infty}$.\footnote{The regularity of $\psi$ follows by that of $\mS$ by classical theory.}
\end{theorem}

Our strategy to prove Theorem \ref{thm:physical} is the following.
We approach the free boundary problem by defining a Lagrangian in terms of the kinetic and potential energies, including the elastic energy in the membrane, in one period of a steady wave (Section \ref{sec:Lagrangian}).
We use  conformal  mappings in a variational setting
\cite{Zakharov et al} to overcome the difficulty that the flow domain
is the unknown (Subsection \ref{subsec:mathematical formulation}).
We then use the direct method of the calculus of variations to maximize the Lagrangian (Section \ref{sec:existence}). 
Key ingredients in the existence theory are Hurwitz's analytical version of the classical isoperimetrical inequality
\cite{chern}, 
which we use to control kinetic and gravitational potential energies in terms of the length of a deformed period of the mebrane, 
and Zygmund's theorem \cite{Zyg} 
for exponential of holomorphic functions on the unit disc, to recover both the non-self-intersection property for the wave profile and compactness.
After that we use the growth hypotheses on the stored energy function to deduce an a priori bounds for the maximizer sufficient to infer that
they satisfy the corresponding Euler-Lagrange system (Section \ref{sec:Euler}).
This is a coupled system of equations, for a pair of periodic function
of a real variable which involves the Hilbert transform, that can be reformulated as a Riemann-Hilbert problem \cite{Sha-Tol,Toland-1998}.
Using this observation it is shown that hydroelastic waves arise from maximizers of the Lagrangian (Section \ref{sec:RH}).
The results in the variational formulation are listed in Theorem \ref{thm:all} (Subsection \ref{subsec:the theorem}), from which Theorem \ref{thm:physical} follows (see Lemma \ref{lemma:maths => phys} and 
Proposition \ref{lemma:reg sigma nu (x)}).

\medskip

We remark that this is a global variational theory; in particular, it is not a theory of small-amplitude solutions.
Its successful application to existence questions is restricted to membranes that are resistant to both bending and stretching: this maximization argument cannot be used for Stokes waves or simple surface tension waves.

\subsection{Illustrative Example}  \label{subsec:example}
The following is a simple illustration of our hypotheses. It shows that our result is valid even for stored elastic energy functions $E(\nu,\mu)$ that include a nontrivial ``mixed term'' of the form $|\mu|^\a / \nu^\d$.

Suppose that $E$ is given by
\[
E(\nu,\mu) = \frac{a}{s}\, \frac{1}{\nu^s}\, + \frac{a}{r}\, \nu^r + b |\mu|^p
+ \b \mu^2 + d\,\frac{|\mu|^\a}{\nu^\d} \, - \frac{a(s+r)}{s\,r}\,,
\]
with $a,b,\b,d,s,\d > 0$, $r>1$, $\a \geq 2$, and $p>2$.

The coefficients of $\nu^r$ and $1/\nu^s$ are such that the minimum of $E(\nu,0)$ occur at $\nu=1$. 
The constant term $- a(s+r)/sr$ guarantees that $E(1,0) = 0$. 
Since $\a$ and $p$ are not less than 2, $E(\nu,\mu)$ is of class $C^2$.  As a consequence, (H\ref{hyp:exists E},\ref{hyp:E=0}) are satisfied. 

(H\ref{hyp:E even mu},\ref{hyp:basic growth},\ref{hyp:E2 >}) and the fact that $E_{11}$ and $E_{22}$ are positive everywhere can be immediately verified. 
If
\begin{subequations}  \label{ill ex}
\begin{equation}  \label{ill ex 1}
\a > \d + 1,
\end{equation}
then the mixed term $|\mu|^\a / \nu^\d$ is strictly jointly convex, and (H\ref{hyp:E jointly convex}) follows. 
We also assume \eqref{ill ex 1}
to prove (H\ref{hyp:uniform infty}).

(H\ref{hyp:Sigma}) holds provided
\begin{equation}  \label{ill ex 2}
c^2 > g + 2 \b_0, 
\end{equation}
where $\b_0 := \b$\, if $\a>2$, 
and $\b_0 := \b+d$\, if $\a=2$.

To prove (H\ref{hyp:special}), it is sufficient to assume that 
\begin{equation}  \label{ill ex 3}
r \geq 4, 
\qquad  \frac{a}{r}\, \geq \,\frac g2\, + g\p + \frac{c^2}{2}\,,
\qquad  b > \frac{a(s+r)}{s\,r}\,,
\end{equation}
by Remark \ref{rem:brutal special} and the continuity of $\mu \mapsto b\mu^p$ near $\mu=1$.

Using Young's inequality \eqref{Young} to control the mixed terms, one can see that (H\ref{hyp:E1 growth control},\ref{hyp:E2 growth control},\ref{hyp:new magic}) hold if 
\[
\a (s+1) + \d p \leq sp.
\]
Note that this inequality, when \eqref{ill ex 1} holds, is implyed by the stronger condition
\begin{equation}  \label{ill ex 4}
\a \leq \frac{p(s+1)}{p+s+1}\,.	
\end{equation}
\end{subequations}
Thus, 
(\ref{ill ex}a,b,c,d) imply (H\ref{hyp:exists E}-\ref{hyp:E2 >}), 
with 
\[
g + 2\b_0 < c^2 \leq \frac{2a}{r}\, - g (1+2\p)
\]
as an interval of admissible velocities. 
A necessary condition for this interval to be nonempty is then
\[
\b_0 < \frac{bs}{s+r}\, - g (1+\p).
\]

\section{The Lagrangian}  \label{sec:Lagrangian}
The strategy for   proving this result  is  to maximize the natural Lagrangian of the physical problem and to observe that such a maximizer yields a non trivial solution of \eqref{classical problem} in which $P$ is given by \eqref{def P}. The Lagrangian involves the fluid's kinetic and potential energies, and the elastic energy of the membrane.  As in  \cite{Zakharov et al,Plotnikov-Toland-nonconvex, Toland-steady, Toland-heavy}, to deal with the unknown flow domain, it is convenient to formulate the Lagrangian using conformal mappings. We begin by considering it in its physical context.

The Lagrangian for travelling waves is the difference between kinetic and potential energies in one period, relative to a frame in which the fluid velocity is stationary.
Formally suppose that one period of the wave profile $\mS$ in the moving frame is given by
\[
\mS_{2\p} = \{ (U(\t), V(\t))\,:\, \t \in [0,2\p] \},
\]
where
\[
U(\t + 2\p) = 2\p + U(\t),
\quad
V(\t + 2\p) = V(\t).
\]
Let $\boldsymbol{U}_{2\p}$ denote one period of the steady flow below $\mS_{2\p}$. Then, in terms of the stream function $\psi$, which satisfies (\ref{classical problem}a,b,c), where $c$ is given, the kinetic energy in one period is
\[
K := \frac12 \int_{\boldsymbol{U}_{2\p}} 
| \gr \big( \psi(X,Y) - c Y \big) |^2 \, dY dX,
\]
the gravitational potential energy is
\[
V_g := \frac g2 \intp V(\t)^2 \, U'(\t) \, d\t,
\]
and, by \eqref{def elastic}, the elastic potential energy is
\begin{equation} \label{Ve elastic}
V_e := \intp E\big( |\br'(x)|, \, |\br'(x)| \, \hat\s(\br(x)) \big)\,dx.
\end{equation}
Note that $V_e$ does not depend on the number $x_0$ that appears in \eqref{r constrain}, because $\br'(x)$ and $\hat\s(\br(x))$ are  $2\p$-periodic functions.
For this reason we fix $x_0=0$ in \eqref{r constrain}.

\begin{remark}  $K$ and $V_g$ are determined by any parametrization of the surface, namely by the \emph{shape} of $\mS$ alone, and not by  the displacement of the material points $x \mapsto \br(x)$ of the undeformed membrane  $\mS$.
By contrast, $V_e$ also depends on both the physical deformation $\br$ and on the shape of $\mS$.
\end{remark}

Thus, the Lagrangian of the travelling waves problem is
\[
\mL = K - V_g - V_e.
\]

\begin{remark}  \label{FBP difficulty}
$K$ involves the solution $\psi$ of a Dirichlet problem 
(\ref{classical problem}a,b,c) on a domain which is itself the main unknown in the problem, and $V_g$ and $V_e$ are integrals on its unknown boundary $\mS$. As a consequence, this is not a Lagrangian in the usual sense, since variations in the domain are involved when discussing  critical points.
In this context we mention a paper of Alt \& Caffarelli \cite{Alt-Caffa} in which a  class of variational free-boundary problems that includes the variational principle for $K-V_g$ is discussed.
\end{remark}

\subsection{The mathematical formulation}
\label{subsec:mathematical formulation}
In \cite{Toland-steady}, following the work of \cite{Zakharov et al} on Stokes waves, the difficulty explained in Remark \ref{FBP difficulty} was overcome by regarding one period of the flow domain as a conformal image of the unit disc, the wave surface being the image of the unit circle.
Here we use the same technology, and we refer to \cite{Sha-Tol,Toland-steady}  for the details.

Let $L^p_{2\p}$ denote the usual Lebesgue space of $2\p$-periodic functions on $\R$, which are $p$-power locally integrable, and $W^{k,p}_{2\p}$ the Sobolev space of $2\p$-periodic functions whose $k$th weak derivative lies in $L^p_{2\p}$, $p \in [1, +\infty]$, $k \in \N$.
Let $[v]$ denote the mean on $[0,2\p]$ of $v \in L^1_{2\p}$.

For any $v \in L^1_{2\p}$, its conjugate function (or Hilbert transform) $\mC v$ from harmonic analysis is defined almost everywhere by
\[
\mC v(\xi) = \frac{1}{2\p} \int_{-\p}^\p \frac{ v(s) }{\tan \frac12\,(\xi-s) }\,ds.
\]
For $p \in (1,+\infty)$ and $k \in \N$,  $\mC$ is a bounded linear operator on $L^p_{2\p}$ and $W^{k,p}_{2\p}$ with
$\mC 1 = 0$ and $\mC \big(e^{int}\big)= -i \,\text{sign} \{n\} \,e^{int}$, $n \neq 0$.

Now, we consider deformations $\br$ such that the shape of the deformed membrane, that is the curve $\{ \br(x) : x \in \R\}$, is $2\p$-periodic in the horizontal direction.
According to constraint ($v$) in Section \ref{physical}, we assume that $\br$ deforms the material points $x$ of any interval of length $2\p$ into one period of the deformed membrane.
Following \cite{Sha-Tol,Toland-steady},  introduce a special parametrization
\[
\rho(w)(\t) := (-\t-\mC w(\t), \, w(\t)),
\quad \t \in \R,
\]
for the curve $\{ \br(x) : x \in \R\}$, where $w(\t)$ is a $2\p$-periodic real function representing
 the elevation of the wave, and $\mC w$ is its Hilbert transform.
\begin{remark}
It is shown in \cite{Toland-steady} that when a curve $\mS(w)$ is defined in terms of $w$ as
\begin{equation}  \label{curve parametrized}
\mS(w) := \{ \rho(w)(\t) : \t \in \R\},
\end{equation}
its slope $\Th(w) (\t)$ and curvature $\s(w)(\t)$ at $\rho(w)(\t)$ are  given by
\[
\Th(w)(\t) := - \mC \log \Om(w)(\t),
\quad \s(w)(\t) := \frac{ \Th(w)'(\t)}{\Om(w)(\t)}\,,
\]
where
\[
\Om(w)(\t) := \sqrt{ (1+\mC w'(\t))^2 + w'(\t)^2 \, }\,.
\]
If $\mS(w)$ is non-self intersecting, it is also shown that
\[
K = \frac {c^2}2 \int_0^{2\pi} w \mC w'\, d\tau
\quad \text{and} \quad
V_g = \frac g2 \int_0^{2\pi} w^2(1+\mC w') \, d\tau.
\]
By \cite{Sha-Tol}, Thm.~2.7, any rectifiable $2\pi$-periodic curve $\mS$  in the plane can be represented as $\mS(w)$
for some $w$ with $w'$ and $\mC w'$ in $L^1_{2\pi}$.
We do not make an {\em a priori}   assumption that $w$ is such that the curve $
\mS(w)
$ is the graph of a function. We will prove that this is so {\em a posteriori}, for maximizers of $J_0$ below. The non-self-intersection property of a curve $\mS$, for given $w$, is a key aspect of this problem.
\end{remark}

To find a formula for $V_e$, we consider diffeomorphisms $\chi(\t)$ of the interval $[0,2\p]$ such that
\[
\chi(0)=0,
\quad
\chi(2\p) = 2\p,
\quad
\chi'(\t)>0 \ \ \text{for a.e.} \ \t,
\]
and
\begin{equation}  \label{x=chi(tau)}
x = \chi(\t)
\quad
\forall \t \in \R.
\end{equation}
In this way, when the surface $\mS$ is defined by $w$, as described above, the position $\br(x)$ of the material point $x$ after the deformation is
\[
\br(x) = \rho(w)(\t),
\]
and the stretch of the membrane is
\begin{equation}  \label{coord nu}
\nu(x) = \frac{|\rho(w)'(\t)|}{\chi'(\t)}\,.
\end{equation}
Note that the curvature $\hat\s(\br(x))$ at  $\br(x)$ depends only on the shape, and not on any particular parametrization,
 of the curve. Then, since $\br(x) = \rho(w)(\t)$,
\begin{equation}  \label{coord sigma}
\hat\s(\br(x)) = \s(w)(\t).
\end{equation}
With the change of variable \eqref{x=chi(tau)}, the elastic energy \eqref{Ve elastic} of the deformation $\br(x)$ has the form
\[ 
\mE(w,\chi) = \intp \chi'(\t) \, E\Big( \frac{\Om(w)(\t)}{\chi'(\t)}\,, \,
\frac{\Om(w)(\t)}{\chi'(\t)} \, \s(w)(\t) \Big) \, d\t.
\] 
Thus formally the hydroelastic wave problem is one of finding critical points for the Lagrangian functional
\[
J (w,\chi) := I(w) - \mE(w,\chi),
\]
where
\[
I(w) := \frac{c^2}{2}\, \intp w' \mC w \, d\t
- \frac{g}{2} \, \intp w^2 (1+\mC w') \, d\t ,
\]
$w$ is a real $2\p$-periodic function belonging to the admissible set $\mathcal A_0$   below, and $\chi$ belongs to
\begin{equation}  \label{def mD}
\mD := \big\{ \chi \in W^{1,1}(\R) : 
\ \chi' \in L^1_{2\p}, 
\ \chi' \geq 0 \  \text{a.e.},
\ \chi(0)=0, 
\ \chi(2\p) = 2\p \big\}.
\end{equation}
If we consider $w = a + \widetilde{w}$, with $a \in \R$ and $[\widetilde{w}]=0$, we see immediately that
\[
\max_{a \in \R} J( a + \widetilde{w}, \chi)
\]
is attained at
\begin{equation}  \label{conservation of mass}
a = - \frac{1}{2\p} \intp \widetilde{w} \mC \widetilde{w}' \, d\t,
\end{equation}
and that this value of $a$ is the one for which the area of the region delimited by the profile $\mS(w)$ and the horizontal axis is 0.
In other words, \eqref{conservation of mass} corresponds to a law of conservation of the mass.
Hence, maximizing  $J$ is equivalent to seeking a maximum of
\[
J_0(w,\chi) := I_0(w) - \mE(w,\chi)
\]
(we have dropped the tilde over $w$), where
\[
I_0(w) := I(w) + \frac{g}{ 4 \p} \, \Big( \intp w' \mC w \, d\t \Big)^2,
\]
with the restriction that $[w]=0$.

\subsubsection*{Remark on $I_0(w)$}

For the existence of the integral in the definition of $I_0$, we need at least that  $w \in \mH^{1,1}_\R$, that is $w$ is a $2\p$-periodic, real, absolutely continuous function with derivative $w' \in L^1_{2\p}$ and $\mC w' \in L^1_{2\p}$ also.
For such functions, $\Om(w) \in L^1_{2\p}$. \qed

\subsubsection*{Remark on $\mE(w,\chi)$}
The integrand of the integral $\mE(w,\chi)$ is defined when the curvature $\s(w)$ of the curve $\mS(w)$ is defined, at least for almost every $\t$.
The formula for the curvature is
\[
\s(w) = \frac{ \Theta(w)' }{ \Om(w) }\,
= - \frac{ 1 }{ \Om(w) }\,
\mC \Big( \frac{ \Om(w)' }{ \Om(w) } \Big)
\]
where $\Theta(w) := - \mC \log \Om(w)$.
The Hilbert transform $\mC$ can be applied to the quotient $\Om'/\Om$ provided it is integrable, that is, when $\log \Om(w) \in W^{1,1}_{2\p}$.
Moreover, $\mC(\Om'/\Om)$ is integrable when $\log \Om(w) \in \mH^{1,1}_\R$ (see above).
In that case, $\log \Om(w)$ is absolutely continuous and periodic, and hence there are two positive constants $a,b$ such that
\[
0 < a \leq \Om(w) \leq b \quad \text{for a.e. }\t .
\]
As a consequence, $\s(w) \in L^1_{2\p}$, since it is the product of a bounded  and an integrable function.

In conclusion, the functional $J_0$ is well-defined for $w$ in the set
\[
\mA_0 := \Big\{ w \in \mH^{1,1}_\R \, : \,[w]=0, \  \log \Om(w) \in \mH^{1,1}_\R,
\  [ \log \Om(w) ] = 0 \,\Big\} ,
\]
which is a subset of the Hardy space $\mH^{1,1}_\R$.
The condition $[\log \Om(w) ] = 0$ is related to the complex formulation of the original water waves problem (see \cite{Toland-steady} and the references therein).

Note that $\mE(w,\chi)$ may be infinite for some $(w,\chi) \in \mA_0 \times \mD$. \qed

\subsection{The theorem in the variational setting}
\label{subsec:the theorem}

The variational problem is the one of finding a maximizer of
\begin{equation}  \label{var problem}
\max\limits_{(w,\chi) \in \mA_0 \times \mD} J_0(w,\chi),
\end{equation}
with $J_0 : \mA_0 \times \mD \to \R \cup \{-\infty\}$,
from which all else follows.
In this variational setting, the complete result is the following.

\begin{theorem}  \label{thm:all}
(a) \emph{(Existence of a maximiser).}
If  \emph{(H\ref{hyp:exists E}-\ref{hyp:special})} hold,
there exists a nontrivial maximizer $(w_0, \chi_0)$ of problem \eqref{var problem}.
Moreover,
\[
w_0 \in \mA_0 \cap W^{2,\rho}_{2\p}, \quad
\chi_0 \in \mD \cap W^{1,s+1} (0,2\p),
\]
where and $p$ and $s$ are the exponents in \emph{(H\ref{hyp:basic growth})} and $\rho := \frac{p+ps}{p+s}\,>1$.

\smallskip

(b) \label{thm:chi' bounded below}
If  \emph{(H\ref{hyp:exists E}-\ref{hyp:uniform infty})} hold, then $\chi'_0 \geq C > 0$ a.e., for some constant $C$.

\smallskip

(c)  \label{thm:Euler chi}
\emph{(Euler equation for the deformation variable $\chi_0$).}
If \emph{(H\ref{hyp:exists E}-\ref{hyp:E2 growth control})} hold,
$\chi_0$ is a solution of the Euler equation \eqref{Euler chi}.

\smallskip

(d)  \label{thm:Euler w}
\emph{(Euler equation for the wave elevation $w_0$).}
If \emph{(H\ref{hyp:exists E}-\ref{hyp:new magic})} hold,
then $w_0$ is a solution of the Euler equation \eqref{Euler}.

\smallskip

(e)  \label{thm:regularity}
\emph{(Regularity).}
If \emph{(H\ref{hyp:exists E}-\ref{hyp:E2 >})} hold, and
$E(\nu,\mu)$ is of class $C^k$, with $k \geq 2$, then
\[
w_0 \in W^{k+1,\b}_{2\p}, \quad \chi_0 \in W^{k,\infty}(0,2\p),
\]
\[
\frac{\Om(w_0)}{\chi'_0} \, \in W^{k-1,\infty}_{2\p},
\quad \s(w_0) \in W^{k-1,\infty}_{2\p}.
\]

\smallskip

(f)  \label{thm:dynamics}
\emph{(Dynamic boundary condition).}
If \emph{(H\ref{hyp:exists E}-\ref{hyp:E2 >})} hold, then $(w_0,\chi_0)$ satisfies the dynamic boundary equation \eqref{dyn final}.
\end{theorem}

We divide the proof of Theorem \ref{thm:all} into distinct parts, introducing Hypotheses only when needed.

\section{Existence theory}  \label{sec:existence}
In this section we prove part $(a)$ of Theorem \ref{thm:all}.
Before anything else, we make some key technical observations.

\subsection{Three technical observations}

\begin{observation}
Among all the curves of length $2 \p \ell$, $\ell>1$, which intersect the horizontal axis at 0 and $2\pi$, the one that achieves the largest possible distance from the horizontal axis is an isosceles triangle.
Therefore
\begin{equation}\label{linf}
\| w \|_\infty \leq \p \sqrt{\ell(w)^2 -1}
\end{equation}
where
\[
\ell(w) = \frac{L(w)}{2\p} \,
= \frac{1}{2\p}\,\intp \Om(w)(\t)\,d\t.
\]
\end{observation}
\begin{observation}
This is based on Hurwitz's analytic version (\cite{chern}, page 29) of the classical isoperimetric inequality:
when
$U,\,V:[a,b] \to \R$ are absolutely continuous with
$U(a)=U(b)$ and $V(a)=V(b)$,
then
\[
\Big|\int_a^b U'(x)V(x) dx \Big|\leq \pi R^2,
\quad \text{ where} \quad
R := \frac{1}{2\pi}\int_a^b \sqrt{U'(x)^2 +V'(x)^2} \, dx.
\]
In other words, $|\int_a^b U'(x)V(x) dx|$ is bounded above by the area of the circle of radius $R$,
and equality holds if and only if $\{(U(x),V(x)):x \in [a,b]\}$ is such a circle.

At this point we refer the reader to Figure \ref{fig1}.
For $\ell >1$,
a circle of radius  $r(\ell)$, where
\[
r(\ell) \sin \vartheta(\ell) = \p
\]
and
\begin{equation}  \label{th ell}
\vartheta(\ell) \in (0,\pi), \quad
\ell = \frac{\vartheta(\ell)}{\sin \vartheta(\ell)}\, ,
\end{equation}
is uniquely determined (up to congruence) by the requirement that the end-points of a chord of length $2\p$ and the end-points of a circular arc of length $2\p \ell$ coincide.
Let $ A(\ell)$ be the area enclosed between the circular arc of length $2 \p \ell$ and the chord of length $2\p$.
Then it is easily seen that
\begin{equation}  \label{A}
A(\ell) = \p^2\,
\frac{2\th(\ell) - \sin (2\th(\ell))}{1 - \cos (2\th(\ell))}
\end{equation}
where $\th(\ell)$ is defined in \eqref{th ell}.
For future convenience, we prove some properties of the function $A(\ell)$.

\begin{lemma}  \label{lemma:A}
$A(\ell)$ is strictly increasing, concave on $(1,\p/2)$ and convex on \\$(\p/2,+\infty)$, and
\begin{equation} \label{A'}
A'(\ell) = \frac{2\p^2}{\sin \th(\ell)}\,.
\end{equation}
Therefore $A'(\ell)>2\p^2$ for all $\ell\neq\p/2$.
Moreover,
\begin{equation} \label{A asymptotics}
\lim_{\ell \to 1} \frac{A(\ell)}{\sqrt{\ell-1}} \to 2\sqrt{2/3}\,\p^2,
\qquad
\lim_{\ell \to +\infty} \frac{A(\ell)}{\ell^2} = \pi,
\end{equation}
$A(\ell)/\sqrt{\ell-1}$ is an increasing function of $\ell$, and
\begin{equation}  \label{A control}
A(\ell) \leq 2\p \,\ell^2
\quad \text{for all }\, \ell > 1\,.
\end{equation}
\end{lemma}

\begin{proof}
First of all, we note that the map
$(1,+\infty) \ni \ell \, \mapsto \, \th(\ell) \in (0,\p)$
is strictly increasing. Indeed,
\begin{equation}  \label{th'}
\th'(\ell) =
\frac{\sin^2\th(\ell)}{ \sin \th(\ell) - \th(\ell) \cos \th(\ell) }\,
>0
\end{equation}
because
\[
\sin \th - \th \cos \th > 0 \quad \forall \th \in (0,\p).
\]
Now,
\[
\frac{d}{d\th} \Big( \frac{2\th-\sin(2\th)}{1-\cos(2\th)} \Big)
= \frac{ 2 (\sin\th - \th \cos\th) }{ \sin^3 \th }
\]
so that \eqref{A'} follows by \eqref{th'}.
Hence $A'(\ell)$ is positive for all $\ell$ because $\th(\ell) \in (0,\p)$.
More, since $\th(\ell)$ is strictly increasing in $\ell$ and $\th(\p/2) = \p/2$, formula \eqref{A'} shows that $A'(\ell)> 2\p^2$ for all $\ell \neq \p/2$, and it is decreasing on $(1,\p/2)$ and increasing on $(\p/2,\infty)$.

Now we note that
\[
\frac{d}{d\ell}\, \frac{A(\ell)}{\sqrt{\ell-1}}\,
= \frac{2A'(\ell) (\ell-1) - A(\ell)}{ 2(\ell-1)^{3/2}} \,.
\]
By \eqref{th ell}, \eqref{A} and \eqref{A'}
\[
2A'(\ell) (\ell-1) - A(\ell)
= \frac{\p^2}{\sin^2 \th(\ell)}\,
\big\{ 3\th(\ell) -4\sin\th(\ell) + \sin\th(\ell) \cos\th(\ell) \big\}
\]
and $3\th -4\sin\th + \sin\th \cos\th > 0$ for all $\th > 0$ because its value at $\th=0$ is zero and its derivative is
\[
\frac{d}{d\th}\,( 3\th -4\sin\th + \sin\th \cos\th ) = 2(1-\cos\th)^2
\geq 0.
\]
Thus, $A(\ell)/\sqrt{\ell-1}$ is an increasing function of $\ell$.

The first limit in \eqref{A asymptotics} can be proved by Taylor series, because $\th(\ell) \to 0$ as $\ell \to1$.
By \eqref{A} and \eqref{th ell},
\begin{equation*}  
\frac{A(\ell)}{ \ell^2}\, =
\p^2 \, \frac{ 2\th(\ell) - \sin (2\th(\ell)) }{ 2\th(\ell)^2} \,,
\end{equation*}
and the second limit in \eqref{A asymptotics} follows because $\th(\ell) \to \p$ as $\ell \to \infty$.

To prove \eqref{A control}, we differentiate
\[
\frac{d}{d\ell}\, \frac{A(\ell)}{\ell^2}\, =
\frac{A'(\ell) \ell - 2 A(\ell)}{\ell^3} \,.
\]
By \eqref{th ell}, \eqref{A} and \eqref{A'}
\[
A'(\ell) \ell - 2 A(\ell) = \frac{2 \p^2 \cos \th(\ell)}{\sin \th(\ell)}\,.
\]
Then $A(\ell)/\ell^2$ has one global maximum at $\ell=\p/2$,
and \eqref{A control} follows because $A(\p/2) = \p^3/2$.
\end{proof}

\begin{proposition} \label{ell}
Suppose that $\{(u(\t),v(\t)): \tau \in [0,2\pi]\}$ is a parametrization of a rectifiable curve of length $2\p \ell$, $\ell > 1$, with $v(0)=v(2\p)$ and $u(2\p)-u(0)=2\p$.
Then
\[
\intp u'(\t)v(\t) d\t \leq A(\ell).
\]
\end{proposition}

\begin{proof}
Suppose that this is false for $(u,v)$ and define  continuous functions $U,\,V$ on the interval $[0,3\pi]$ as follows.
Let $(U,V)$ coincide with $(u,v)$ on $[0,2\pi]$,
let $(U(3\pi),V(3\pi)) = (u(0),v(0))$,
and let $\{(U(x),V(x)): x \in [2\pi,3\pi]\}$ be an injective parametrization of the arc of the circle with radius $r(\ell)$ which is complementary to $c(\ell)$.
Therefore, by the divergence theorem,
\[
\int_{2\pi}^{3\pi} U'(x)V(x) \, dx = \pi r(\ell)^2 - A(\ell).
\]
Since the proposition is supposed to be false, we find from the definition of $(U,V)$ that
\[
\int_0^{3\pi} U'(x)V(x) \,dx > \pi r(\ell)^2,
\]
where, by construction,
\[
\int_0^{3\pi} \sqrt{U'(x)^2 +V'(x)^2} \, dx = 2\pi r(\ell).
\]
This contradicts the isoperimetric inequality and proves the result.
\end{proof}
\end{observation}
\begin{observation}
The third observation is the vector version of Jensen's inequality:
if $f:\R^n \to \R$ is convex, $U \subset \R^m$ has unit measure and $u:U \to \R^n$, then
\begin{equation}  \label{jensen}
\int_U f(u(x)) \, dx \geq f \Big( \int_U u(x) \, dx \Big).
\end{equation}
This is immediate from the fact (\cite{E&T}, Ch. I.3) that $f:\R^n \to \R$ is convex  if and only if
\[
f(x) = \sup \big\{ a \cdot x +b \,:
\ a \in \R^n, \ b \in \R, \ a\cdot y+b \leq f(y) \ \, \forall y \in \R^n  \big\}.
\]
\end{observation}

\subsection{Estimates of $\Theta(w)$ when $J_0(w,\chi) \geq 0$}
We seek a positive maximum of $J_0$ on $\mA_0 \times \mD$.

\begin{lemma} \label{lemma:Sigma>0}
Suppose that \emph{(H\ref{hyp:exists E},\ref{hyp:E even mu},\ref{hyp:E=0},\ref{hyp:Sigma})} hold.
Then
\[
\Sigma := \sup_{(w,\chi) \in \mA_0 \times \mD} J_0(w,\chi) > 0 .
\]
\end{lemma}

\begin{proof}
Let $w_\e(\t) := \e \cos \t$ and $\chi'_\e(\t) := 1 + \e \cos \t$.
Then $\mC w_\e' = \e \cos \t$, and calculations with Taylor series for $\e \to 0$ give
\[
I_0 (w_\e) = \frac\p2 \, (c^2-g) \, \e^2 + O(\e^4),
\]
\[
\frac{\Om(w_\e)}{\chi'_\e}\, = 1 + O(\e^2),
\quad
\frac{\Om(w_\e)\, \s(w_\e)}{\chi'_\e}\, =  - \,\e \cos \t + O(\e^2).
\]
By (H\ref{hyp:exists E},\ref{hyp:E even mu},\ref{hyp:E=0}), the Taylor series of $E(\nu,\mu)$ near $(1,0)$ is
\[
E(\nu,\mu) = \frac{E_{11}(1,0)}{2}\, (\nu-1)^2 + \frac{E_{22}(1,0)}{2}\, \mu^2 + o( (\nu-1)^2  + \mu^2 ).
\]
Hence
\[
\chi'_\e \, E\Big( \frac{\Om(w_\e)}{\chi'_\e}\,,
\frac{\Om(w_\e)\, \s(w_\e)}{\chi'_\e} \Big)
= \frac{E_{22}(1,0)}{2}\, \e^2 \cos^2 \t + o(\e^2),
\]
and
\[
\mE(w_\e, \chi_\e) = \frac\p2 \, E_{22}(1,0) \, \e^2 + o(\e^2).
\]
Thus (H\ref{hyp:Sigma}) implies that $J_0(w_\e,\chi_\e) >0$ for all $\e \neq0$ sufficiently small.
\end{proof}

\begin{lemma}  \label{lemma:compactness}
Suppose that \emph{(H\ref{hyp:exists E},\ref{hyp:E even mu},\ref{hyp:E jointly convex},\ref{hyp:special})} hold and $\mu^* <1$ is that defined in \emph{(H\ref{hyp:special})}.
Then, for all $(w,\chi)\in \mA_0 \times \mD$ such that $J_0(w,\chi) > 0$,
\[
| \Theta(w) (\t_1) - \Theta(w) (\t_2) | < \mu^* \p \quad \forall \t_1, \t_2 \in \R .
\]
\end{lemma}

\begin{proof}
First, we consider the elastic energy $\mE(w,\chi)$, which is
\[
\mE(w,\chi) = \intp \chi' \, E\Big( \frac{\Om(w)}{\chi'}\,, \,
\frac{\Om(w) \, |\s(w)|}{\chi'} \Big) \, d\t
\]
by (H\ref{hyp:E even mu}).
From (H\ref{hyp:E jointly convex}),
and Jensen's inequality \eqref{jensen} it follows that
\begin{equation} \label{V>e} 
\mE(w,\chi)
\geq
2\p \, E \big( \ell(w), m(w) \big),
\end{equation}
where
\[
\ell(w) := \frac{1}{2\p} \intp \Om(w)(\t) \, d\t,
\quad
m(w) := \frac1{2\p} \intp |\Theta(w)'(\tau)| \, d\t .
\]

Now we consider
\[
I_0(w) = \frac{g}{4\p} \Big( \intp w \mC w' d\t \Big)^2 +
\frac{c^2}{2} \, \intp w  \mC w' d\t
- \frac{g}{2} \, \intp w^2 (1+\mC w') \, d\t .
\]
Since $w$ is $2\p$-periodic and has zero mean, by Proposition \ref{ell},
\[
0 \leq \intp w \mC w' d\t = \intp w (1+\mC w') d\t \leq A(\ell(w) ),
\]
because the length of a period of $\mS(w)$ is $2\p\ell(w)$.
Integrating by parts shows that
\begin{align*}
\intp w^2 \mC w' d\t
& = \intp w ( w \mC w' - \mC (ww')) \, d\t + \intp w \mC (w w') \, d\t \\
& = \intp w ( w \mC w' - \mC (ww')) \, d\t + \frac12 \intp w^2 \mC w' d\t .
\end{align*}
Since  $w \mC w' - \mC (w w')\geq 0$ almost everywhere (see \cite{Toland-1998}, Prop. 3.1),
\[
\Big| \intp w^2 \mC w' d\t \Big| \leq 2 \| w \|_{\infty}\intp  ( w \mC w' - \mC (w w')) \, d\t
= 2 \| w \|_{\infty} \intp w \mC w' d\t,
\]
because $\mC (w w')$ has zero mean.
From these inequalities and \eqref{linf},
\begin{equation}  \label{I0}
I_0(w) \leq \frac{g}{4\p } A(\ell(w) )^2 +\frac{c^2}{2}A(\ell(w)) +g\pi A(\ell(w) )\sqrt{\ell(w)^2 - 1} \, .
\end{equation}
Let $J_0(w,\chi) > 0$. Then $\mE(w,\chi) < I_0(w)$, whence, by \eqref{V>e} and \eqref{I0},
\[
E (\ell, m ) < \frac{g}{8\p^2 } A(\ell )^2 +\frac{c^2}{4\p}A(\ell) +\frac g2 A(\ell )\sqrt{\ell^2-1},
\]
where $\ell = \ell(w)$ and $m=m(w)$.
Now, suppose, for contradiction, that $m \geq \mu^*$.
Since,
by (H\ref{hyp:E even mu}) and (H\ref{hyp:E jointly convex}),
$\mu \mapsto E(\nu,\mu)$ is increasing on $[0,+\infty)$ for every fixed $\nu$,
\[
E ( \ell ,\mu^* )
\leq
E(\ell, m )
<
\frac{g}{8\p^2 } A(\ell )^2 + \frac{c^2}{4\p}A(\ell)
+ \frac g2 A(\ell) \sqrt{\ell^2-1},
\]
violating (H\ref{hyp:special}).
Thus, $m < \mu^*$.
Therefore, for every $\t_1 < \t_2 < \t_1 + 2\p$,
\begin{align*}
|\Theta(w)(\t_1) - \Theta(w)(\t_2)|
& \leq
\frac12 \, \Big| \int_{\t_1}^{\t_2} \Theta(w)'(\t)\, d\t - \int_{\t_2}^{\t_1 + 2\p} \Theta(w)'(\t) \,d\t \Big| \\
& \leq
\frac12 \intp |\Theta(w)'(\t)| \,d\t
= \p m
< \mu^* \p < \p.  \qedhere
\end{align*}
\end{proof}

\subsection{Maximising sequences}
Assume (H\ref{hyp:Sigma}), so that $\Sigma >0$ by Lemma \ref{lemma:Sigma>0}, and let $\{ (w_k, \chi_k )\} \subset \mA_0 \times \mD$ be a maximising sequence, with
\[
0 < J_0 (w_k, \chi_k) \to \Sigma \quad (k \to \infty).
\]
For convenience, let $\Om_k := \Om(w_k)$, $\s_k := \s(w_k)$,
$\Theta_k := \Theta(w_k)$.
By Lemma \ref{lemma:compactness} we have that
\[
| \Theta_k (\t_1) - \Theta_k (\t_2) | < \mu^* \p
\quad \forall \t_1, \t_2 \in \R,
\]
and the bound is uniform for $k \in \N$.
Let
\[
M_k := \frac12 \, \Big( \max_{\t \in [0,2\p]} \Theta_k(\t) + \min_{\t \in [0,2\p]} \Theta_k(\t) \Big)
\]
and set
\[
\widetilde \Theta_k (\t) := \Theta_k (\t) - M_k .
\]
Then $\| \widetilde \Theta_k \|_\infty < \mu^* \p /2$. Moreover, since
$\Theta_k = - \mC \log \Om_k$ and $[\log \Om_k] = 0$,
\[
\mC \widetilde \Theta_k = \mC  \Theta_k = \log \Om_k.
\]
We recall Zygmund's Theorem \cite{Zyg}, Vol.~I, page 257: If $f \in L^\infty_{2\p}$, then
\[
\intp \exp \{ q | \mC f (\t)| \} \, d\t
\leq \frac{4\p}{ \cos ( q \| f \|_\infty ) }
\quad \text{for all } 0 \leq q < \frac{ \p }{ 2 \| f \|_\infty } \,.
\]
We apply this result to $f = \widetilde \Theta_k$ and we get
\[
\intp \Om_k^{\,q} \, d\t
\leq \intp e^{ q |\log \Om_k|}\, d\t
< \frac{ 4\p}{ \cos ( q \mu^* \p /2 )}
\quad \text{for all }  0 \leq q < \frac{1}{\mu^*} \, .
\]
Since $\mu^*<1$, in what follows we can fix an exponent $q\in (1,1/\mu^*)$
and obtain a uniform bound 
\begin{equation}  \label{Om Lq}
\| \Om_k \|_{L^q_{2\p}} \leq C \quad \forall k
\end{equation}
for the maximising sequence.

By \eqref{Om Lq} and H\"older's inequality we get a uniform bound for $\Om_k$ in $L^1_{2\p}$,
\[
L_k := L(w_k) = \intp \Om_k \, d\t \leq C
\quad \forall k,
\]
where here, and more generally in this section, $C$ denotes positive, possibly different constants.
By \eqref{I0} $I_0(w)$ is dominated by a function of $L(w)$. Then, since   $J_0(w_k, \chi_k) > 0$, we have a uniform bound for the elastic energy
\begin{equation}  \label{elastic bound}
\mE(w_k,\chi_k) < I_0(w_k) \leq C
\quad \forall k .
\end{equation}
Since $[\chi'_k] = 1$ for all $k$,
from (H\ref{hyp:basic growth}) it follows that
\begin{equation}  \label{power bound}
\intp \Big\{
\Big( \frac{\Om_k}{\chi'_k} \Big)^{r}
+ \Big( \frac{\chi'_k}{\Om_k} \Big)^s
+ \Big( \frac{\Om_k \,|\s_k|}{\chi'_k}\Big)^p
\Big\} \, \chi'_k \, d\t \,
\leq C
\quad \forall k.
\end{equation}
Now, let
\[
\a:= 1 + \frac sq\,,
\quad \b:= \frac s\a\,,
\quad \g := \frac{s+1}{\a}\,.
\]
By H\"older's inequality,
\[
\intp (\chi'_k)^\g \, d\t
\leq \Big( \intp \Big( \frac{(\chi'_k)^\g}{\Om_k^\b} \Big)^\a \, d\t \Big)^{\frac{1}{\a}}
\Big( \intp (\Om_k^\b)^{\a'}\,d\t \Big)^{\frac{1}{\a'}} .
\]
Then, by \eqref{power bound} and \eqref{Om Lq},
\begin{equation}  \label{chi' Lg}
\| \chi'_k \|_{L^\g_{2\p}} \leq C \quad \forall k,
\end{equation}
because $\g \a = s+1$, $\b\a=s$ and $\a'\b=q$.
Note that $\g>1$ because $q>1$.

Now let
\[
\lm:= 1 + \frac{p-1}{\g}\,, \quad
\eta := \frac{p}{\lm}\,, \quad
\xi := \frac{p-1}{\lm}\,.
\]
By H\"older's inequality,
\[
\intp |\Th'_k|^\eta \, d\t
\leq
\Big( \intp \Big( \frac{|\Th'_k|^\eta}{(\chi'_k)^\xi} \Big)^\lm \, d\t \Big)^{\frac{1}{\lm}}
\Big( \intp (\chi'_k)^{\xi\lm'}\,d\t \Big)^{\frac{1}{\lm'}} .
\]
Then, by \eqref{power bound} and \eqref{chi' Lg}, and recalling that
$\Om_k \, \s_k=\Th_k'$,
\[
\| \Th'_k \|_{L^\eta_{2\p}} \leq C \quad \forall k,
\]
because $\eta \lm = p$, $\xi\lm=p-1$ and $\xi\lm'=\g$.
Note that $\eta>1$ because $\g>1$.
Since $[ \Theta_k ] = [-\mC \log \Omega_k]=0$,
\[
\| \Theta_k \|_{W^{1,\eta}_{2\p}} \leq C
\quad \forall k.
\]
It follows that
\[
\| \mC \Theta_k \|_{W^{1,\eta}_{2\p}}
= \| \log \Om_k \|_{W^{1,\eta}_{2\p}} \leq C
\quad \forall k
\]
and then $\log \Om_k$ are absolutely continuous functions with a uniform bound
$\| \log \Om_k \|_\infty \leq C$.
This means that there are two constants $C,C' $ such that
\begin{equation}  \label{lower upper bound}
0 < C \leq \Om_k(\t) \leq C' \quad \forall \t, \ \forall k.
\end{equation}
Thanks to \eqref{lower upper bound}, the bound \eqref{power bound} becomes
\begin{equation}  \label{power bound after}
\intp \Big(
\frac{1}{(\chi'_k)^{r-1}}\,
+ (\chi'_k)^{s+1}
+ \frac{|\s_k|^p}{(\chi'_k)^{p-1}}
\Big) \, d\t
\leq C
\quad \forall k.
\end{equation}
Thus \eqref{chi' Lg} can be improved to
\[ 
\| \chi'_k \|_{L^{s+1}_{2\p}} \leq C \quad \forall k.
\] 
Then, since $\chi_k(0)=0$ for all $k$,
\[
\| \chi_k \|_{W^{1,s+1}(0,2\p)} \leq C \quad \forall k.
\footnote{We write $W^{1,s+1}(0,2\p)$, and not $W^{1,s+1}_{2\p}$, because the diffeomorphism $\chi_k$, unlike its derivative $\chi'_k$, is not a periodic function; see \eqref{def mD}.}
\]
Also, let
\[
a:= 1 + \frac{p-1}{s+1}\,, \quad
b:= \frac{p-1}{a}\,, \quad
\rho := \frac pa\,= \frac{p(s+1)}{s+p}\,.
\]
By H\"older's inequality,
\[
\intp |\s_k|^\rho \, d\t
\leq
\Big( \intp \Big( \frac{|\s_k|^\rho}{(\chi'_k)^b} \Big)^a \, d\t \Big)^{\frac{1}{a}}
\Big( \intp (\chi'_k)^{b a'}\,d\t \Big)^{\frac{1}{a'}} .
\]
Then, by \eqref{power bound after},
\[ 
\| \s_k \|_{L^\rho_{2\p}} \leq C \quad \forall k,
\] 
because $\rho a = p$, $ba = p-1$ and $ba'= s+1$.
Note that $\rho>\eta>1$ by construction.
By the last inequality and \eqref{lower upper bound},
\[
\| \Theta_k' \|_{L^\rho_{2\p}} \leq C \quad \forall k
\]
and then $\| \mC \Theta_k' \|_{L^\rho} \leq C$,
\[
\| \Om_k' \|_{L^\rho_{2\p}} \leq C \quad \forall k.
\]
Since $w_k' = \Om_k \sin \Theta_k$ and $\| \Om_k \|_\infty \leq C$, by the previous bound we get
\begin{equation}  \label{wk''}
\| w_k'' \|_{L^\rho_{2\p}}
\leq C \quad \forall k,
\quad  \rho = \frac{p+sp}{p+s}\,,
\end{equation}
and so $w_k$ is a bounded sequence in $W^{2,\rho}_{2\p}$,  because $[w_k] = 0$ for all $k$.

\subsection{Weak limit $(w_0,\chi_0)$ of the maximising sequence}

By \eqref{wk''}, and since $[w_k]=0$, there exists a subsequence, which we denote $(w_k)$ as well, and a function $w_0 \in W^{2,\rho}_{2\p}$ such that $w_k \rightharpoonup w_0$ in $W^{2,\rho}_{2\p}$ weakly; more precisely:
\begin{align*}
& w_k \rightarrow w_0, 
\quad \mC w_k \rightarrow \mC w_0
\quad \, \text{in} \ L^\infty_{2\p} \ \text{strongly}, \vspace{4pt}\\
& w_k' \rightarrow w_0', 
\quad \mC w_k' \rightarrow \mC w_0'
\quad \, \text{in} \ L^\infty_{2\p} \ \text{strongly}, \vspace{4pt}\\
& w_k'' \rightharpoonup w_0'', 
\quad \mC w_k'' \rightharpoonup \mC w_0'' 
\quad \text{in} \ L^\rho_{2\p}  \ \text{weakly}.
\end{align*}
In particular, $\Om_k \to \Om_0 := \Om(w_0)$ uniformly and, as a consequence of \eqref{lower upper bound}, $\Om_0$ is bounded below, so that
\begin{equation}  \label{Om uniform}
\Om_k \to \Om_0  
\quad \text{and} \quad 
\log \Om_k \to \log \Om_0 
\quad \text{in} \ L^\infty_{2\p}.	
\end{equation}
It follows that $[\log \Om_0 ] = [w_0]=0$ because $[\log \Om_k]=[w_k]=0$ for all $k$.
Thus $w_0 \in \mA_0$.
Since $\mC w_k' \to \mC w_0'$ and $w_k' \to w_0'$  in $L^\infty_{2\p}$, it follows that
\[
\Om_k' \rightharpoonup \Om_0' 
\quad \text{and} \quad
\frac{ \Om_k' }{ \Om_k } \, \rightharpoonup \frac{\Om_0'}{\Om_0}\, 
\quad \text{in} \ L^\rho_{2\p}  \ \text{weakly}
\]
and so
\[
\Theta_k' = - \mC \Big( \frac{ \Om_k' }{ \Om_k } \, \Big) \,
\rightharpoonup - \mC \Big( \frac{\Om_0'}{\Om_0}\, \Big)\, = \Theta_0'
 \quad \text{in} \ L^\rho_{2\p}  \ \text{weakly}\,.
\]
Thus $\Theta_k' / \Om_k \rightharpoonup \Theta_0' / \Om_0$ weakly, that is
\begin{equation}  \label{sigma weak convergence}
\s_k \rightharpoonup \s_0
 \quad \text{in} \ L^\rho_{2\p}  \ \text{weakly}	
\end{equation}
where $\s_0 := \s(w_0)$.

Moreover, a subsequence $(\chi_k)$ converges to some $\chi_0 \in W^{1,s+1}(0,2\p) \cap \mD$,
\begin{align}  \label{chi weak convergence}
& \chi_k \rightarrow \chi_0
\quad \text{in} \ L^\infty(\R) \ \text{strongly}, \\ 
\notag
& \chi_k' \rightharpoonup \chi_0' \quad \text{in} \ L^{s+1}_{2\p}  \ \text{weakly}.
\end{align}

Obviously
\[
\intp w_k \mC w_k' \, d\t
\to \intp w_0 \mC w_0' \, d\t
\]
and
\[
\intp w_k^2 (1 + \mC w_k') \, d\t
\to
\intp w_0^2 (1 + \mC w_0') \, d\t,
\]
so that $I_0(w_k) \to I_0 (w_0)$.

\subsection{The existence of a maximum}
To prove that $(w_0, \chi_0)$ is a maximiser for $J_0$, it is more convenient to write the elastic energy as
\[
\mE(w,\chi) = \intp \Om(w) \, E^\star \Big( \frac{\chi'}{\Om(w)}\,, \s(w) \Big) \, d\t,
\]
where
\[
E^\star(t,\s) := t E\Big( \frac{1}{t}\,, \frac{\s}{t} \Big),
\quad \forall t > 0, \ \s\in \R.
\]
Note that $(E^\star)^\star=E$, that is
\[
E(\nu,\mu) = \nu E^\star\Big( \frac{1}{\nu}\,, \frac{\mu}{\nu} \Big),
\quad \forall \nu > 0, \ \mu\in \R.
\]
$E^\star$ is jointly convex in both its argument by (H\ref{hyp:E jointly convex}) and the following lemma. Note that $E^\star$ coincides with $\widetilde{\mathbf e}$ in \cite{Toland-steady}, Remark 3.1.

\begin{lemma} \label{lemma:E* convex}
$E^\star(t,\s)$ is jointly convex in $(t,\s)$ if and only if
$E(\nu,\mu)$ is jointly convex in $(\nu,\mu)$.
\end{lemma}

\begin{proof}
Differentiating gives
\[
E_{22}(\nu,\m) = \frac1\nu \, E^\star_{22}\Big( \frac{1}{\nu}\,, \frac{\nu}{\mu} \Big)
\]
and
\[
E_{11}(\nu,\mu) \, E_{22}(\nu,\mu) - E_{12}(\nu,\mu)^2
= \frac{1}{\nu^4}
\Big\{ E^\star_{11} \Big( \frac{1}{\nu}\,, \frac{\nu}{\mu} \Big) \,
E^\star_{22}\Big( \frac{1}{\nu}\,, \frac{\nu}{\mu} \Big)
- E^\star_{12} \Big( \frac{1}{\nu}\,, \frac{\nu}{\mu} \Big)^2 \Big\}.
\]
Then the jointly convexity of $E$, that is (H\ref{hyp:E jointly convex}), implies that
$E^\star_{11}$, $E^\star_{22}$ and $E^\star_{11} E^\star_{22} - (E^\star_{12})^2$ are positive at every $(t,\s)$ with $t > 0$ and $\s \in \R$.

The opposite is true because $E=(E^\star)^\star$.
\end{proof}

Now, by \eqref{chi weak convergence}, $\chi'_k \rightharpoonup \chi'_0$ in $L^\rho_{2\p}$ weakly because $\rho < s+1$.
Therefore, since $1/\Om_k \to 1/\Om_0$ uniformly, 
\[
\frac{\chi_k'}{\Om_k}\, \rightharpoonup
\frac{\chi_0'}{\Om_0} \quad \text{in} \ L^\rho_{2\p} \ \text{weakly},
\]
and so, by \eqref{sigma weak convergence}, the pairs $(\chi'_k/\Om_k, \s_k)$ converge to $(\chi'_0/\Om_0, \s_0)$ weakly in the product space $L^\rho_{2\p} \times L^\rho_{2\p}$.
We define
\[
F : L^\rho_{2\p} \times L^\rho_{2\p} \to \R,
\quad
F(u,v) := \intp \Om_0 \, E^\star(u,v) \, d\t.
\]
Since $E^\star$ is continuous and non-negative
(recall that $E \geq 0$ by (H\ref{hyp:E=0})),
by Fatou's Lemma $F$ is strongly lower semicontinuous on $L^\rho_{2\p} \times L^\rho_{2\p}$. 
Moreover, by (H\ref{hyp:E jointly convex}) and Lemma \ref{lemma:E* convex}, $F$ is also convex.
Hence $F$ is weakly lower semicontinuous on $L^\rho_{2\p} \times L^\rho_{2\p}$. It follows that
\[
\mE(w_0, \chi_0) = F\Big( \frac{\chi'_0}{\Om_0}\,,\s_0 \Big)
\leq \liminf_k F\Big( \frac{\chi'_k}{\Om_k}\,,\s_k \Big)
= \liminf_k \intp \Om_0 \, E^\star \Big( \frac{\chi'_k}{\Om_k}\,,\s_k \Big) \, d\t.
\]
We note that
\begin{align*}
\Big| F\Big( \frac{\chi'_k}{\Om_k}\,,\s_k \Big) - \mE(w_k,\chi_k) \Big|
& = \Big| \intp \frac{(\Om_0 - \Om_k)}{\Om_k} \, \Om_k \, E^\star \Big( \frac{\chi'_k}{\Om_k}\,,\s_k \Big) \, d\t \Big| \\
& \leq C \, \| \Om_0 - \Om_k \|_{L^\infty}
\, \to 0
\end{align*}
as $k \to \infty$ by \eqref{lower upper bound}, \eqref{elastic bound} and \eqref{Om uniform}.
Hence
\[
\liminf_k F\Big( \frac{\chi'_k}{\Om_k}\,,\s_k \Big)
= \liminf_k \mE(w_k,\chi_k).
\]
On the other hand, by the definition of the maximising sequence $(w_k,\chi_k)$,
\[
I_0(w_k) - \mE(w_k,\chi_k) = J_0(w_k, \chi_k) \to \Sigma,
\]
and, since $I_0(w_k) \to I_0(w_0)$,
\[
\mE(w_k,\chi_k) \to I_0(w_0) - \Sigma .
\]
Then
\begin{equation}  \label{auxiliary}
\mE(w_0, \chi_0)
\leq \liminf_k \mE(w_k,\chi_k)
= I_0(w_0) - \Sigma,
\end{equation}
that is, $J_0(w_0,\chi_0) \geq \Sigma$. Therefore
\[
J_0(w_0,\chi_0) = \Sigma
\]
and $(w_0,\chi_0)$ is a maximum for $J_0$ on $\mA_0 \times \mD$.

We have proved that, when (H\ref{hyp:exists E}-\ref{hyp:special}) hold, there exists a maximizer $(w_0, \chi_0)$ for problem \eqref{var problem}.
Since $J_0(0,1)=0$ and $J_0(w_0,\chi_0) = \Sigma>0$, the maximum is nontrivial. Moreover, we have also proved that
\[
w_0 \in \mA_0 \cap W^{2,\rho}_{2\p}, \quad
\chi_0 \in \mD \cap W^{1,s+1} (0,2\p),
\]
and the proof of part ($a$) of Theorem \ref{thm:all} is concluded.

\section{Euler equation and regularity of the solution}
\label{sec:Euler}
We next prove parts ($b$), ($c$) and ($d$) of Theorem \ref{thm:all}.
First, to see that $\chi_0'$ is bounded below, we note the following facts about $E^\star$.

\begin{lemma}  \label{lemma:aux}
Suppose that \emph{(H\ref{hyp:E even mu},\ref{hyp:E jointly convex},\ref{hyp:basic growth})} hold.
Then, for all $\s \in \R$,
\[
\lim_{t \to 0^+} E^\star_1(t,\s) = -\infty, \quad
\lim_{t \to +\infty} E^\star_1(t,\s) = +\infty.
\]
\end{lemma}

\begin{proof}
By (H\ref{hyp:E even mu},\ref{hyp:E jointly convex}) and Lemma \ref{lemma:E* convex}, $E^\star(t,\s)$ is even in $\s$ and jointly convex, therefore
\begin{equation}  \label{sure sign}
E^\star(t,\s) \geq E^\star(t,0) \quad \forall t>0, \ \s \in \R.
\end{equation}
Since $E^\star(t,0) = t E(1/t,0)$,
by (H\ref{hyp:basic growth}) we have that
\[
\lim_{t \to 0^+} E^\star(t,0) = + \infty,
\quad
\lim_{t \to +\infty} \frac{E^\star(t,0)}{t}\,  = + \infty.
\]
Then, by \eqref{sure sign},
\[
\lim_{t \to 0^+} E^\star(t,\s) = + \infty,
\quad
\lim_{t \to +\infty} \frac{E^\star(t,\s)}{t}\,  = + \infty
\]
for all $\s \in \R$. The lemma easily follows by the convexity of the map $t \mapsto E^\star(t,\s)$, for every fixed $\s$.
\end{proof}

We have proved that $E_1^\star(t,\s) \to -\infty$ as $t \to 0^+$.
(H\ref{hyp:uniform infty}) is equivalent to assuming that such a limit is uniform in $\s$.
Indeed, we observe that, for every $\nu$,
\[
\inf_{\mu \in \R} \big( \gr E(\nu,\mu) \cdot (\nu,\mu) - E(\nu,\mu) \big)
= \inf_{\s \in \R} \big( \gr E(\nu,\s\nu) \cdot (\nu,\s\nu) - E(\nu,\s\nu) \big).
\]
Hence (H\ref{hyp:uniform infty}) holds if and only if
\[
\lim_{t \to 0^+} \{ \sup_{\s \in \R} E^\star_1(t,\s) \}
= -\infty,
\]
that is
\begin{equation}  \label{comoda}
E^\star_1(t,\s) \leq f(t) \quad \forall t>0, \ \s \in \R,
\quad \text{and} \quad
\lim_{t \to 0^+} f(t) = -\infty,
\end{equation}
where $f(t) := \sup_{\s} E^\star_1(t,\s)$.

As a consequence of \eqref{comoda}, there exists a positive constant $\a^*$ such that
$f(t) < 0$ for all $t \in (0,\a^*)$, therefore
\begin{equation}  \label{old alpha*}
E^\star_1(t,\s) < 0 \quad \forall t \in (0,\alpha^*)\,, \ \s \in \R.
\end{equation}

Note that the map $t \mapsto E^\star_1(t,\s)$ is strictly increasing. Then, by Lemma \ref{lemma:aux}, for every $\s$ there exists a unique $t^\star(\s)>0$ such that $E^\star_1(t^\star(\s),\s)=0$, and $t^\star(0)=1$.
By \eqref{old alpha*}, $t^\star$ is bounded below, namely $t^\star(\s) \geq \alpha^*$ for all $\s$.

\begin{lemma}  \label{lemma:chi bound below}
Let $(w_0,\chi_0)$ be the maximum for $J_0$ described in part ($a$) of Theorem \ref{thm:all}.
Suppose that \emph{(H\ref{hyp:exists E},\ref{hyp:E jointly convex},\ref{hyp:uniform infty})} hold.
Then there exists a constant $C$ such that
\[
\chi'_0(\t) \geq C >0 \quad \text{for a.e. } \t.
\]
\end{lemma}

\begin{proof}
Suppose that $1/\chi'_0$ does not belong to $L^\infty$.
Then all the sets
\[
A_n := \Big\{ \t \in (0,2\p) : \, \chi'_0(\t) \leq \frac{1}{n} \Big\},
\quad n \in \N,
\]
have positive Lebesgue measure denoted by $|A_n|>0$.
Note also that
\[
\{ \t \in (0,2\p) : \, \chi'_0(\t) \geq 1 \}
\]
has positive measure --- if not, then $\chi'_0 <1$ almost everywhere, whence
\[
\chi_0(2\p) - \chi_0(0) = \intp \chi'_0(\t)\,d\t < 2\p,
\]
violating the fact that $\chi_0 \in \mD$. 
Since $\sigma_0$ and $\chi_0'$ are integrable we can therefore choose $N$ large enough that
\[
B:=\{ 1 \leq \chi'_0 < N, \ |\s_0|<N \}
\]
has positive measure.
Then, for every $n$ we define $\ph_n$ by
\begin{align*}
\ph_n(\t) := \begin{cases}
2/n & \text{if } \t \in A_n, \\
\chi'_0(\t)-\lm_n & \text{if } \t \in B, \\
\chi'_0(\t) & \text{everywhere else},
\end{cases}
\end{align*}
where
\[
\lm_n := \frac{1}{|B|} \int_{A_n} \Big( \frac{2}{n}\, - \chi'_0 \Big) \, d\t ,
\]
so that
\[
\intp \ph_n(\t) \, d\t = 2\p
\]
for all $n$. Also, we note that
\begin{equation}  \label{lmn}
\frac{1}{|B|}\,\frac{|A_n|}{n}\, \,\leq\, \lm_n \,\leq\,
\frac{2}{|B|}\,\frac{|A_n|}{n} 
\end{equation}
for all $n$. We define
\[
\widetilde{\chi}_n(\t) := \int_0^\t \ph_n(\tilde\t)\,d\tilde\t
\]
and observe that $\widetilde{\chi}_n \in \mD$ for all $n$ sufficiently large.
We calculate the difference
\[
\mE(w_0,\chi_0) - \mE(w_0,\tchi_n)
 = \intp \Om_0 \, \Big\{ E^\star \Big( \frac{\chi'_0}{\Om_0}\,,\s_0 \Big)
- E^\star\Big( \frac{\ph_n}{\Om_0}\,,\s_0 \Big) \Big\}\,d\t
= a_n + b_n,
\]
where
\[
a_n := \int_{A_n} \Om_0\,\Big\{ E^\star\Big( \frac{\chi'_0}{\Om_0}\,,\s_0\Big)
- E^\star\Big( \frac{2}{\Om_0\,n}\,,\s_0 \Big) \Big\}\,d\t
\]
and
\[
b_n:= \int_B \Om_0\,\Big\{ E^\star\Big( \frac{\chi'_0}{\Om_0}\,,\s_0\Big)
- E^\star\Big( \frac{\chi'_0 - \lm_n}{\Om_0}\,,\s_0 \Big) \Big\}\,d\t.
\]
Since $\Om_0(\t) \geq C > 0$ for all $\t$,
\[
\frac{\chi'_0(\t)}{\Om_0(\t)}\, \leq \frac{1}{\Om_0(\t)\,n}\,
< \frac{2}{\Om_0(\t)\,n} \leq \frac{C'}{n} < \alpha^*
\quad \forall \t \in A_n,
\]
for all $n$ sufficiently large, where $\alpha^*$ is defined in
\eqref{old alpha*}, using (H\ref{hyp:uniform infty}).
Then, by \eqref{old alpha*} and the fact that $E^\star_{11}>0$,
\begin{align*}
E^\star\Big( \frac{\chi'_0}{\Om_0}\,,\s_0\Big)
- E^\star\Big( \frac{2}{\Om_0\,n}\,,\s_0 \Big)
& > - E^\star_1 \Big( \frac{C'}{n}\,, \s_0 \Big)
\, \Big( \frac{2}{\Om_0\,n} - \frac{\chi'_0}{\Om_0} \Big) \\
& \geq - f \Big( \frac{C'}{n} \Big) \, \frac{1}{\Om_0\,n}
\quad \forall \t \in A_n,
\end{align*}
where $f$ is defined in \eqref{comoda}.
Hence
\[
a_n > - f \Big( \frac{C'}{n} \Big) \, \frac{|A_n|}{n}\,.
\]

To estimate $b_n$, we observe that $\chi'_0/\Om_0$ and $(\chi'_0-\lm_n)/\Om_0$ are confined in a compact interval $K$ which does not contain zero, for all $\t \in B$, for all $n$ sufficiently large.
Hence we define
\[
M:= \max_{t \in K, \ |\s| \leq N} \,|E^\star_1(t,\s)|,
\]
and so
\[
|b_n| \leq \int_B M \lm_n \, d\t = |B|\,M \lm_n.
\]
Then, by \eqref{lmn},
\[
\frac{a_n}{|b_n|}\, > - \frac{f(C'/n)}{2 \,M} \, \to +\infty
\]
as $n \to \infty$ by \eqref{comoda}. This implies that
\[
a_n + b_n > 0
\]
for $n$ sufficiently large, so that $\mE(w_0,\tchi_n) < \mE(w_0,\chi_0)$, violating the maximality of $(w_0,\chi_0)$ for $J_0$.
\end{proof}

\subsection{The Euler equation for $\chi_0$}

Now we prove part ($c$) of Theorem \ref{thm:all}.
By \eqref{auxiliary}, $\mE(w_0,\chi_0) < \infty$, therefore, 
by (H\ref{hyp:basic growth}) and the fact that $0<C\leq \Om_0 \leq C'$, 
we know that
\[
\intp \frac{|\s_0|^p}{(\chi'_0)^{p-1}}\, d\t < \infty.
\]
Since $\chi'_0$ is bounded below (Lemma \ref{lemma:chi bound below}),
\[
\intp \Big( \frac{|\s_0|}{\chi'_0} \Big)^p \, d\t
\leq \Big\| \frac{1}{\chi'_0} \Big\|_\infty \,
\intp \frac{|\s_0|^p}{(\chi'_0)^{p-1}}\, d\t < \infty,
\]
that is $\s_0 / \chi'_0$ belongs to $L^p_{2\p}$.
Moreover, recall that
\[
\chi_0' \in L^{s+1}_{2\p}, \quad
E\Big( \frac{\Om_0}{\chi'_0}\,, \frac{\Om_0 \s_0}{\chi'_0} \Big) \in L^1_{2\p}.
\]

To study the differentiability of the functional $J_0(w,\chi)$ with respect to $\chi$, we assume (H\ref{hyp:E1 growth control},\ref{hyp:E2 growth control}).
(H\ref{hyp:E1 growth control}) implies that
\[
E_{1,0} :=
E_1 \Big( \frac{\Om_0}{\chi'_0} \,, \frac{\Om_0\,\s_0}{\chi'_0} \Big)
\in L^1_{2\p},
\]
and (H\ref{hyp:E2 growth control}) implies that
\[
E_{2,0} :=
E_2 \Big( \frac{\Om_0}{\chi'_0} \,, \frac{\Om_0\,\s_0}{\chi'_0} \Big)
\in L^{p'}_{2\p},
\]
where $1/p + 1/p' = 1$, because $\Om_0 \,\s_0 / \chi'_0 \in L^p_{2\p}$,
$\chi_0'/\Om_0 \in L^{s+1}_{2\p}$ and $\Om_0/\chi'_0 \in L^\infty_{2\p}$.
As a consequence, the functional $\mE(w,\chi)$ is Gateaux-differentiable with respect to $\chi$, and its partial derivative in any direction $\psi \in W^{1,\infty}_{2\p}$, with $\chi_0 + \psi \in \mD$, at $(w_0,\chi_0)$ is
\begin{multline*}
d_\chi \mE(w_0,\chi_0) \, \psi
\\ = \intp \Big\{
E\Big( \frac{\Om_0}{\chi'_0}\,, \frac{\Om_0\,\s_0}{\chi'_0} \Big)
 -
\gr E\Big( \frac{\Om_0}{\chi'_0}\,, \frac{\Om_0\,\s_0}{\chi'_0} \Big)
\cdot
\Big( \frac{\Om_0}{\chi'_0}\,, \frac{\Om_0\,\s_0}{\chi'_0} \Big)
 \Big\} \, \psi' \, d\t ,
\end{multline*}
where
\[
\gr E(\nu,\mu) \cdot (\nu,\mu) = \nu \, E_1(\nu,\mu) + \mu\,E_2(\nu,\mu).
\]
Then the maximiser $(w_0,\chi_0)$ satisfies the Euler-Lagrange equation for the functional $J_0(w,\chi)$ with respect to $\chi$, that is
\begin{equation} \label{Euler chi}
E\Big( \frac{\Om_0}{\chi'_0}\,, \frac{\Om_0\,\s_0}{\chi'_0} \Big)
 -
\gr E\Big( \frac{\Om_0}{\chi'_0}\,, \frac{\Om_0\,\s_0}{\chi'_0} \Big)
\cdot
\Big( \frac{\Om_0}{\chi'_0}\,, \frac{\Om_0\,\s_0}{\chi'_0} \Big)
\equiv \g_0,
\end{equation}
for some constant $\g_0 \in \R$.
Note that \eqref{Euler chi} is equation (1.2) of \cite{Toland-steady}.

We have proved that, if $(w_0,\chi_0)$ is the maximum of $J_0$ described in part ($a$) of Theorem \ref{thm:all}, with $\chi'_0(\t) \geq C > 0$, and (H\ref{hyp:basic growth},\ref{hyp:E1 growth control},\ref{hyp:E2 growth control}) hold, then $(w_0,\chi_0)$ solves the Euler-Lagrange equation \eqref{Euler chi}. Thus part ($c$) of Theorem \ref{thm:all} is proved.

\subsection{The Euler equation for $w_0$}
\label{subsec:Euler w}

We prove part ($d$) of Theorem \ref{thm:all}.
The maps
\[
W^{1,\rho}_{2\p} \to \R,
\quad w \mapsto \intp w \mC w' \, d\t \quad \text{and}
\quad w \mapsto \intp w^2 \mC w' \, d\t
\]
are Fr\'echet differentiable at $w_0$.
Hence $I_0$ is differentiable on $W^{1,\rho}_{2\p}$, and its differential at $w_0$ in the direction $h \in W^{1,\rho}_{2\p}$ is
\begin{equation*}  
d I_0(w_0) h = \intp \gr I_0 \, h \, d\t,
\end{equation*}
where
\begin{equation}  \label{grad I0}
\gr I_0 := \gr I_0(w_0) =
( c^2 -2 g a_0) \mC w_0'
- g \big\{ w_0 (1+\mC w_0') + \mC (w_0 w_0') \big\}
\end{equation}
and
\begin{equation}  \label{def alpha 0}
a_0 := - [w_0 \mC w_0'].
\end{equation}
In the following, we will denote
\[
H^{k,\rho}_0:= \{ w \in W^{k,\rho}_{2\p} \,: \ [w]=0 \},
\quad k=1,2.
\]
To investigate the differentiability (at least in the Gateaux sense)
of the elastic energy term $\mE(w,\chi)$ with respect to $w$,
we recall the following fact.
\begin{lemma} \label{lemma:neighbourhood}
$w_0$ is an interior point of $\mA_0$ in the topology of $W^{2,\rho}_{2\p}$, that is there exists $\e_0 = \e_0(w_0) >0$ such that
\[
\mW_0 := \{ w \in H^{2,\rho}_{0}\,:  \ \|w-w_0\|_{W^{2,\rho}_{2\p}} < \e_0 \} \subset \mA_0,
\]
and there exist constants $C,C'$ such that
\[
0 < C \leq \Om(w) \leq C'
\quad \forall w \in \mW_0.
\]
\end{lemma}

\begin{proof}
See \cite{Toland-steady}, Lemma 4.1.
\end{proof}

The map
\[
\Om : \mW_0 \to W^{1,\rho}_{2\p} \,,
\quad
w \mapsto \Om(w)
\]
is of class $C^1$, and its differential at $w$ in the direction $h$ is
\[
d \Om(w) h = \frac{ w' h' + (1+\mC w') \mC h'}{ \Om(w) }
\quad \forall h \in H^{2,\rho}_0.
\]
Also the map
\[
\Om \s :  \mW_0 \to L^\rho_{2\p} \,,
\quad
w \mapsto  \Om(w) \s(w)
\]
is of class $C^1$, and
\[
d (\Om(w) \s(w)) h = - \mC \Big( \frac{d \Om(w)h}{\Om(w)} \Big)'
= -\mC \Big( \frac{ w' h' + (1+\mC w') \mC h'}{ \Om(w)^2 } \Big)',
\]
because, by definition, $\Om(w) \s(w) = \Th(w)'$ and $\Theta(w) = - \mC \log \Om(w)$.
We define
\[
\mL(u) := \frac{ w_0' u + (1+\mC w_0')\,\mC u}{ \Om_0^2} \,,
\]
so that
\[
d \Om(w_0) h = \Om_0 \, \mL(h'),
\quad
d (\Om(w_0) \s(w_0)) h = - \, \mC (\mL(h'))'.
\]
Note that, in general, $\mL(h')$ belongs to $W^{1,\rho}_{2\p}$ only, even for $h \in C^\infty$.

The functional $\mE(w,\chi)$ is Gateaux-differentiable with respect to $w$ at $(w_0,\chi_0)$ in the direction $h$ if
\[
E_{1,0} \, \mL(h') - E_{2,0} \, \mC (\mL(h'))' \in L^1_{2\p}.
\]
Since $\mL(h') \in W^{1,\rho}_{2\p}$ and $E_{1,0} \in L^1_{2\p}$,
this holds when
\begin{equation}  \label{when E2}
E_{2,0} \in L^{\rho'}_{2\p},
\end{equation}
where
\[
\rho' = \frac{p(s+1)}{s(p-1)}
\]
is the conjugate exponent of $\rho=(p+sp)/(p+s)$. Now
\eqref{when E2} holds if we assume (H\ref{hyp:new magic}).
Indeed, recalling that $\Om_0$ is bounded both below and above,
(H\ref{hyp:new magic}) and the Euler equation \eqref{Euler chi} imply that, where $|\s_0|/\chi_0'$ is larger than some constant (depending on the constant $\g_0$ of \eqref{Euler chi}), then
\[
\Big( \frac{|\s_0|}{\chi'_0} \Big)^p
\leq C (\chi'_0)^s,
\]
for some $C>0$. Hence
\[
\frac{\s_0}{\chi'_0}\, \in L^{\frac{p(s+1)}{s}}
\]
because $\chi'_0 \in L^{s+1}$.
Then \eqref{when E2} follows by (H\ref{hyp:E2 growth control}), and, as a consequence, the functional $\mE(w,\chi)$ is Gateaux-differentiable at $(w_0,\chi_0)$ with respect to $w$ in all directions $h \in W^{2,\rho}_{2\p}$.
Hence the maximiser $(w_0,\chi_0)$ solves the Euler equation in weak form
\begin{multline}  \label{weak Euler}
\intp \big\{ \gr I_0\, h
- E_{1,0} \Om_0 \, \mL(h')
+ E_{2,0} \, \mC \mL(h')' \big\} \, d\t = 0
\quad \forall h \in H^{2,\rho}_0.
\end{multline}

\begin{lemma}  \label{lemma:inv mL}
The linear operator $\mL$ is an isomorphism of $H^{1,\rho}_{0}$ into itself, and
\[
\mL(u) = v \quad \text{iff} \quad u = w_0' v - (1+\mC w_0') \, \mC v =: \mL^{-1}(v),
\]
for all $u,v \in H^{1,\rho}_{0}$.
\end{lemma}
\begin{proof} 
The proof is elementary once the following facts from complex function theory are taken into account. 
For $p>0$, write  $U\in \mH^p_{\CC}$ if $U$ is holomorphic in the unit disc $D$ and
\[
\sup_{r \in (0,1)} \int_0^{2\pi}|U(re^{i\tau})|^p \, d\tau < \infty.
\] 
It is well known \cite{duren, garnett} that when $U \in \mH^p_{\CC}$ for any $p>0$, then
\[
U^*(t):=\lim_{r\nearrow 1 }U(re^{it}) \ \text{ exists for a.e.}\ t
\]
and if, for some  $q \in (0,\infty)$,
$|U^*|^q \in L^1_{2\pi}$, then $U\in\mH^q_{\CC}$;  if $|U^*| \in L^\infty_{2\pi}$, then  $U$ is bounded on $D$.
 Moreover, if $U \in \mH^p_{\CC}$, $p>0$, then $U \in \mH^q_{\CC}$ for some $q>1$ if and only if   $U^*(t) = u(t)+i\mC u(t) + i\alpha$ for some $u \in L^q_{2\pi}$ where
\[
\frac{1}{2\pi}\int_0^{2\pi} u(\tau) \, d\tau + i \alpha = U(0).
\] 
Conversely, if $u \in L^q_{2\pi}$, $q>1$, there exists $U \in \mH^q_{\CC}$ with $U^* = u+i\mC u$.

Finally (\cite{duren}, Theorem 3.11),  
if $U$ is holomorphic on $D$ and continuous on $\overline D$, then $U^*$ is absolutely continuous if and only if $U' \in \mH^1_\CC$, in which case
\[
\frac{d~}{d\tau}\, U^*(\tau) = i e^{i\tau} (U')^*(\tau).
\] 
Moreover, if $U' \in \mH^1_\CC$ and $[\log {U'}^*]= \log |U'(0)|$, then $U'$ has no zeros in $D$.

Now we turn to our proof. 
Since $w'_0 \in H^{1,\rho}_{0}$, there exists a function $W$ which is holomorphic in the unit disc $D$,
continuous on the closed unit disc $\overline D$, $W(0) = i$,  
$W' \in \mH^\rho_\CC$ and  
\begin{equation}  \label{def W}
W^* = w_0'+i(1+\mC w_0'), 
\,\quad 
|W^*| = \Omega(w_0).
\end{equation}
Since $w_0 \in \mA$, it follows that $W$ and $1/W$ are bounded on  $\overline D$.

Now for $u \in H^{1,\rho}_0$, let $U$ be holomorphic on $D$ and continuous on $\overline D$ such that $U' \in \mH^\rho_\CC$ and $U^* = u+i \mC u$. Note that $\mL u = \Re (U/W)^*$. It follows from the above remarks that $(U/W)'\in \mH^\rho_\CC$. 
Hence $\mL u$ is absolutely continuous and $(\mL u)'\in L^\rho_{2\pi}$. Also, since $W(0)\neq 0$, it follows that  $[\mL u] =0$ if and only if $U(0) =0$, i.e. if and only if $[u]=0$. 
If, on the other hand, $0=\mL u = \Re (U/W)^*$, then $U \equiv 0$ and hence $\mL$ is injective from $H^{1,\rho}_{0}$ to itself. 

Finally for any $v \in H^{1,\rho}_{0}$ let $V$ be holomorphic on $D$ and continuous on $\overline D$ such that $V' \in \mH^\rho_\CC$ and $V^* = v+i \mC v$. 
Then $V(0)=0$ and $\mL u = v$ if and only if $U = WV$ on $D$. 
Thus $U' \in \mH^\rho_\C$ and $U(0)=0$. 
In other words, $u \in H^{1,\rho}_0$ and  
$u = w_0' v - (1+\mC w_0') \, \mC v $. 
This completes the proof.
\end{proof}

\begin{remark} \label{rem:obstacle}
In Lemma \ref{lemma:inv mL} we have proved that $\mL$ is an isomorphism of $H^{1,\rho}_{0}$.
However, $h \in H^{2,\rho}_{0}$ does not imply that $\mL(h)\in H^{2,\rho}_{0}$, 
because $w_0'$ has only regularity $W^{1,\rho}$.
This means that, in the present problem, $\mL(h)$ cannot be taken as test function,  as was done in \cite{Toland-steady}.
So instead, here we will take $\mL(h') \in H^{1,\rho}_0$ as test function ``of lower order''.
\end{remark}

Now we seek an expression for  \eqref{weak Euler} that involves only $\mL(h')$ as test function.
First, we note that, for every $h \in H^{2,\rho}_0$,
\[
\intp \gr I_0 \, h \, d\t = \intp (\gr I_0 - \lm_0)\, h \, d\t,
\]
where $\lm_0 := [\gr I_0]$.
Integrating by parts yields
\begin{align*}
\intp (\gr I_0 - \lm_0) \, h \, d\t
& = - \intp \Big( \int_0^\t (\gr I_0 - \lm_0) \Big) h'(\t) \, d\t \\
& = - \intp m_0\, \mL(h') \, d\t ,
\end{align*}
where
\[
m_0 := (\mL^{-1})^* \Big( \int_0^\t (\gr I_0 - \lm_0) \Big)
\]
and $(\mL^{-1})^*$ is the adjoint operator of $\mL^{-1}$ in the usual $L^2_{2\p}$ sense, 
\[
(\mL^{-1})^*(f) = w_0' f + \mC \big( (1+\mC w_0') f \big)
\quad \forall f.
\]
Now, for any $h \in H^{2,\rho}_0$, let
\[
\ph := \mL(h').
\]
Since $[h']=0$, by Lemma \ref{lemma:inv mL}, $\ph \in H^{1,\rho}_0$.
We observe that
\[
\big\{ \ph = \mL(h'): \ h \in H^{2,\rho}_0 \big\}
= H^{1,\rho}_0.
\]
Indeed, given any $\ph \in H^{1,\rho}_0$, there exists a unique primitive $h$ of $\mL^{-1}(\ph)$ having zero mean.
As a consequence, \eqref{weak Euler} can be written as
\begin{equation}  \label{weak Euler ph}
\intp ( m_0 + \Om_0 E_{1,0} ) \, \ph \, d\t
\, + \, \intp \mC E_{2,0} \, \ph' \, d\t
\, = \, 0
\qquad \forall \ph \in H^{1,\rho}_0,
\end{equation}
where $\mC (E_{2,0})$ is well defined, by \eqref{when E2}.

\begin{lemma}  \label{lemma:weak ode}
Suppose that $a(\t)\in L^{\rho'}_{2\p}$, $b(\t) \in L^1_{2\p}$ satisfy
\[
\intp b \ph \, d\t \, + \intp a \ph' \, d\t \, = \, 0
\quad \forall \ph \in H^{1,\rho}_0.
\]
Then $a(\t) \in W^{1,1}_{2\p}$, and
\[
a(\t) = \text{const.} + \int_0^\t \big( b(t) - [b] \big) \, dt.
\]
\end{lemma}

\begin{proof}
The proof is elementary.
\end{proof}

By Lemma \ref{lemma:weak ode} and \eqref{weak Euler ph} we deduce the Euler equation
\begin{equation}  \label{Euler}
\mC E_{2,0} (\t) = \text{const.} + \int_0^\t ( m_0 + \Om_0 E_{1,0} - b_0 ) \, dt,
\end{equation}
where $b_0 := [m_0 + \Om_0 E_{1,0}]$.

We have proved that, if the maximizer $(w_0,\chi_0)$ of $J_0$ described in part ($a$) of Theorem \ref{thm:all} satisfies the Euler equation \eqref{Euler chi}, with $\chi'_0(\t) \geq C > 0$, and (H\ref{hyp:E2 growth control},\ref{hyp:new magic}) hold, then $(w_0,\chi_0)$ also solves the Euler equation \eqref{Euler}. In Section \ref{sec:RH}, this will be shown to imply the dynamic boundary condition \eqref{psi dynamics}. But first we examine the smoothness of solutions.

\subsection{Regularity of the solution}
\label{subsec:regularity}

We prove part ($e$) of Theorem \ref{thm:all}.
From \eqref{Euler} it follows that $\mC E_{2,0} \in W^{1,1}_{2\p} \subset L^\infty \subset L^\b_{2\p}$ for all $\b \in (1,\infty)$. Hence
\begin{equation}  \label{E2 L beta} 
E_{2,0} \in L^\b \quad \forall \b \in (1,\infty).
\end{equation}
In  (H\ref{hyp:E1 growth control}-\ref{hyp:E2 >}),
let $\bar\nu_4 := \min \{ \bar\nu_1, \bar\nu_2, \bar\nu_3, \bar\nu_{\g_0} \}$ and $\bar\mu_4 := \max \{ \bar\mu_1, \bar\mu_2, \bar\mu_3, \bar\mu_{\g_0} \}$,
where $\g_0$ is the constant in the Euler equation \eqref{Euler chi}, and let
\[
A^* := \Big\{ \t\in (0,2\p) : \
\frac{\Om_0}{\chi'_0}\, \leq \bar\nu_4,
\quad
\frac{\Om_0 |\s_0|}{\chi'_0}\, \geq \bar\mu_4 \Big\}.
\]
By (H\ref{hyp:new magic},\ref{hyp:E2 >}), 
\begin{align*}
\int_{A^*} |E_{2,0}|^\b \, d\t &
\geq C \int_{A^*} \Big( \frac{\Om_0}{\chi'_0} \Big)^{\a\b}\,
\Big( \frac{\Om_0 |\s_0|}{\chi'_0} \Big)^{(p-1)\b}\,d\t  \\
& \geq C' \int_{A^*} (\chi'_0)^{\e\b}\,d\t.
\end{align*}
Note that $\chi'_0$ is bounded above on $(0,2\p) \setminus A^*$ by (H\ref{hyp:new  magic}) and Euler equation \eqref{Euler chi}.
Hence, by \eqref{E2 L beta}, $\chi'_0 \in L^\b_{2\p}$ for all $\b \in (1,\infty)$, and therefore 
\[
\frac{\s_0}{\chi'_0} \in L^\b_{2\p} \quad \forall \b \in (1,\infty),
\]
by (H\ref{hyp:new  magic}). 
Thus, by (H\ref{hyp:E1 growth control}),
\[
|E_{1,0}| \leq C \Big\{ (\chi'_0)^{s+1} + \Big( \frac{|\s_0|}{\chi'_0} \Big)^p \Big\}
\in L^\b_{2\p} \quad \forall \b \in (1,\infty),
\]
for some $C$. Since $m_0 \in W^{1,\rho}_{2\p} \subset L^\infty_{2\p}$, by \eqref{Euler} it follows that
$\mC E_{2,0} \in W^{1,\b}_{2\p}$ for all $\b$, therefore
\begin{equation} \label{E2,0 W1beta}
E_{2,0} \in W^{1,\b}_{2\p} \quad \forall \b \in (1,\infty).
\end{equation}
In particular, $E_{2,0} \in L^\infty_{2\p}$.
Hence, by (H\ref{hyp:new magic},\ref{hyp:E2 >}), 
\[
\| E_{2,0} \|_\infty \geq C \frac{1}{(\chi'_0)^\a}\,
\Big( \frac{|\s_0|}{\chi'_0}\Big)^{p-1}
\geq C' (\chi'_0)^\e , \quad  \text{on }A^*.
\]
Thus
\[
\chi'_0 \in L^\infty_{2\p}.
\]
By (H\ref{hyp:new magic}), also $\s_0/\chi'_0 \in L^\infty_{2\p}$, therefore
\[
\s_0 \in L^\infty_{2\p}.
\]
By the definition of $E^\star(t,\s)$, the Euler equation \eqref{Euler chi} for $\chi$ can be written as
\begin{equation}  \label{Euler* chi}
E^\star_1 \Big( \frac{\chi'_0}{\Om_0}\,, \s_0 \Big) \equiv \g_0.
\end{equation}
By (H\ref{hyp:E jointly convex}), the map $t \mapsto E^\star_1(t,\s)$ is strictly increasing. 
Hence, by Lemma \ref{lemma:aux}, it follows that for every $\s \in \R$ there exists a unique $t>0$ such that $E^\star_1(t,\s) = \g_0$. 
In other words, there is defined a function $\vp : \R \to (0,+\infty)$ such that
\begin{equation}  \label{def vp}
E^\star_1 (\vp(\s),\s) = \g_0
\quad \forall \s \in \R.
\end{equation}
Since $E(\nu,\mu)$ is of class $C^2$ (H\ref{hyp:exists E}), $\vp$ is of class $C^1$ by the implicit function theorem, and
\begin{equation}  \label{vp'}
\vp'(\s) = - \frac{E^\star_{12}(\vp(\s),\s)}{E^\star_{11}(\vp(\s),\s)}\,.
\end{equation}
We rewrite \eqref{Euler* chi} as
\begin{equation}  \label{Euler vp chi}
\frac{\chi'_0}{\Om_0}\, = \vp(\s_0),
\end{equation}
therefore
\begin{equation}  \label{E*2 vp}
E_{2,0}
= E_2 \Big( \frac{\Om_0}{\chi'_0}\,,  \frac{\Om_0 \s_0}{\chi'_0} \Big)
= E_2 \Big( \frac{1}{\vp(\s_0)}\,, \frac{\s_0}{\vp(\s_0)} \Big)
= E^\star_2 (\vp(\s_0), \s_0)
\end{equation}
by the definition of $E_{2,0}$ and $E^\star$.
Now we consider the map
\[
\psi : \R \to \R,
\quad
y \mapsto \psi(y) := E^\star_2 (\vp(y),y).
\]
$\psi$ is differentiable and
\[
\psi'(y) =
\frac{E^\star_{11}(\vp(y),y) \, E^\star_{22}(\vp(y),y) - (E^\star_{12}(\vp(y),y))^2}{E^\star_{11}(\vp(y),y)}
\]
by \eqref{vp'}. Thus, by (H\ref{hyp:E jointly convex}) and  Lemma \ref{lemma:E* convex}, $\psi$ is strictly increasing, and therefore invertible.
Hence \eqref{E*2 vp} can be written as
\begin{equation}  \label{sigma0 = psi -1 (E2,0)}
\s_0 = \psi^{-1} (E_{2,0}) .
\end{equation}
By \eqref{E2,0 W1beta}, $\psi^{-1} (E_{2,0})$ is differentiable, and
\[
\partial_\t \{ \psi^{-1} ( E_{2,0}(\t) ) \}
= \frac{1}{ \psi'(\s_0(\t))} \, (E_{2,0})'(\t),
\]
for almost all $\t$.
Since $\psi'$ is positive and continuous, and $|\s_0(\t)| \leq \| \s_0 \|_{L^\infty}$ for almost all $\t$,
\[
\psi'(\s_0) \geq C > 0 \quad \text{a.e.},
\]
for some constant $C$.
Therefore, by \eqref{E2,0 W1beta}, $\psi^{-1}(E_{2,0})$ belongs to $W^{1,\b}_{2\p}$ for all $\b$, that is,
\begin{equation}  \label{sigma0 W1beta}
\s_0 \in W^{1,\b}_{2\p}
\quad \forall \b \in (1,\infty).
\end{equation}
By the usual sequence of (non-trivial) implications, one can then prove that
\[
\Om_0 \in W^{2,\b}_{2\p}, \quad
w_0 \in W^{3,\b}_{2\p} \quad \forall \b \in (1,\infty).
\]
By the fact that
\begin{equation}  \label{compact comfort}
0 < C \leq \vp(\s_0) = \frac{\chi'_0}{\Om_0}\, \leq C'
\end{equation}
and by \eqref{sigma0 W1beta} and \eqref{vp'}, it follows that $\vp(\s_0) \in W^{1,\b}_{2\p}$ for all $\b$. Then by \eqref{Euler vp chi}
\[
\chi_0 \in W^{2,\b}_{2\p}   \quad \forall \b \in (1,\infty).
\]
Hence $E_{1,0}$ is differentiable, and, by \eqref{compact comfort}, \eqref{sigma0 W1beta} and the continuity of the second derivatives $E_{11}$ and $E_{12}$, we have proved that
\[
E_{1,0} \in W^{1,\b}_{2\p}   \quad \forall \b \in (1,\infty).
\]
Next, the fact that $w_0 \in W^{3,\b}_{2\p}$ implies that
\[
m_0 \in W^{2,\b}_{2\p}  \quad \forall \b \in (1,\infty)
\]
by the definition of $m_0$ and $\gr I_0$.
Then, by \eqref{Euler},
\begin{equation}  \label{E2,0 W2beta}
E_{2,0} \in W^{2,\b}_{2\p}  \quad \forall \b \in (1,\infty).
\end{equation}
In particular, $(E_{2,0})'$ is bounded.
Since $\s'_0 = (E_{2,0})' / \psi'(\s_0)$, we conclude that
\[
\s_0 \in W^{1,\infty}_{2\p}.
\]
Hence, by \eqref{compact comfort}, 
\[
\frac{\Om_0}{\chi'_0}\, \in W^{1,\infty}_{2\p}.
\]
Moreover, since $\Om_0$ and $1/\Om_0$ belong to $W^{1,\infty}_{2\p}$, also
\[
\frac{\chi'_0}{\Om_0}\,, \ \chi'_0, \ \frac{1}{\chi'_0}\, \in W^{1,\infty}_{2\p}.
\]

Differentiating \eqref{Euler chi} with respect to $\t$ (which is possible now because $\Om_0/\chi'_0$ and $\s_0$ are differentiable) yields
\begin{equation}  \label{old Q trick}
(E_{1,0})' + \s_0 (E_{2,0})' = 0,
\end{equation}
which is \eqref{palp}.
By \eqref{E2,0 W2beta} and \eqref{sigma0 W1beta},
\eqref{old Q trick} implies that
\[
E_{1,0} \in W^{2,\b}_{2\p}.
\]

We have proved that, if $E(\nu,\mu)$ is of class $C^2$, then the curvature $\s_0$ and the stretch $\Om_0/\chi'_0$ of the membrane belong to $W^{1,\infty}_{2\p}$.
By bootstrap, when $E(\nu,\mu)$ enjoys more regularity, $\s_0$ and $\Om_0/\chi'_0$ are also more regular, as the following result shows.

\begin{lemma} \label{lemma:bootstrap} Suppose that 
\emph{(H\ref{hyp:exists E},\ref{hyp:E even mu},\ref{hyp:E jointly convex},\ref{hyp:basic growth},\ref{hyp:E1 growth control},\ref{hyp:new magic},\ref{hyp:E2 >})} hold, and let $(w_0,\chi_0)$ satisfy Euler equations \eqref{Euler chi} and \eqref{Euler}.
Suppose that $E(\nu,\s)$ is of class $C^k$, $k \geq 2$.
Then 
\[
w_0 \in W^{k+1,\b}_{2\p}, \quad \chi_0 \in W^{k,\infty}_{2\p},
\quad \s_0 \in W^{k-1,\infty}_{2\p},
\quad
\frac{\Om_0}{\chi'_0} \in W^{k-1,\infty}_{2\p},
\]
\[
E_{1,0} \in W^{k,\b}_{2\p},
\quad
E_{2,0} \in W^{k,\b}_{2\p},
\]
for all $\b \in (1,\infty)$.
\end{lemma}

\begin{proof}
We have already proved the case $k=2$.
By induction, suppose that the statement holds for all $j = 2, \ldots, k$, and let $E(\nu,\mu) \in C^{k+1}$.
Hence $\Om_0$ and $m_0$ belong to $W^{k,\b}_{2\p}$, because $w_0 \in W^{k+1,\b}_{2\p}$. Then, by the Euler equation \eqref{Euler} for $w_0$,
\begin{equation}  \label{induz.1}
E_{2,0} \in W^{k+1,\b}_{2\p}
\end{equation}
for all $\b$.
Recalling the definition of $\vp$ and $\psi$, we note that, since $E \in C^{k+1}$, both $\vp$ and $\psi$ are of class $C^k$.
Then, with a direct calculation, one can see that the $k$th derivative of $\psi^{-1}(E_{2,0})$ with respect to $\t$ is a finite sum of terms, each of which is a quotient where the numerator is a polynomial involving the partial derivatives of $E$ at $(\Om_0/\chi'_0$, $\Om_0 \s_0/\chi'_0)$ of order $\leq (k+1)$, and the denominator is an integer power of $\psi'(\s_0)$.
Since $E \in C^{k+1}$, by \eqref{compact comfort} and the fact that $\s_0$ is bounded it follows that the $k$th derivative of $\psi^{-1}(E_{2,0})$ is bounded. Thus, by \eqref{sigma0 = psi -1 (E2,0)},
\begin{equation}  \label{induz.2}
\s_0 \in W^{k,\infty}_{2\p}.
\end{equation}
Using \eqref{old Q trick}, \eqref{induz.1} and \eqref{induz.2} imply that
\[
E_{1,0} \in W^{k+1,\b}_{2\p}
\]
for all $\b$.

Next, with a direct calculation, one can see that the $k$th derivative of $\vp(\s_0)$ with respect to $\t$ is a polynomial involving the derivatives of $\vp$ at $\s_0$ of order $\leq k$, and the derivatives of $\s_0$ at $\t$ of order $\leq k$. 
Since $\vp \in C^k$, and by \eqref{induz.2}, the $k$th derivative of $\vp(\s_0)$ is bounded. 
By the Euler equation \eqref{Euler vp chi} for $\chi_0$,
\[
\frac{\Om_0}{\chi'_0} \in W^{k,\infty}_{2\p}.
\]
Moreover, since $\Om_0/\chi'_0 \geq C > 0$ for some constant $C$, also
\begin{equation}  \label{induz.5}
\frac{\chi'_0}{\Om_0} \in W^{k,\infty}_{2\p}.
\end{equation}

By \eqref{induz.2}, with the usual sequence of implications, it follows that
\[
\Om_0 \in W^{k+1,\b}_{2\p}, \quad
w_0 \in W^{k+2,\b}_{2\p}
\]
for all $\b$. Hence $\Om_0 \in W^{k,\infty}_{2\p}$, therefore, by \eqref{induz.5},
\[
\chi'_0 = \frac{\chi'_0}{\Om_0}\, \Om_0 \in W^{k,\infty}_{2\p},
\]
and the proof is complete.
\end{proof}

\begin{remark}
\emph{As an obvious consequence of Lemma \ref{lemma:bootstrap}, if $E(\nu,\s) \in C^\infty$, then
$w_0$,
$\chi_0$,
$\s_0$,
$\Om_0/\chi'_0$,
$E_{1,0}$ and
$E_{2,0}$
are also $C^\infty$.}
\end{remark}

By using the inverse diffeomorphism $\chi_0^{-1}$,
and recalling \eqref{x=chi(tau)}, \eqref{coord nu} and \eqref{coord sigma},
the same regularity result holds for the stretch
\[
\nu_0(x) := \frac{\Om_0(\chi_0^{-1}(x))}{\chi'_0(\chi_0^{-1}(x))}\, =
\frac{\Om_0(\t)}{\chi'_0(\t)}
\]
and the curvature
\[
\hat\s_0(\br(x)) := \s_0(\chi_0^{-1}(x)) = \s_0(\t)
\]
as functions of the Lagrangian coordinate $x$ of material points.

\begin{proposition}  \label{lemma:reg sigma nu (x)}
Under the same assumptions as in Lemma \ref{lemma:bootstrap},
\[
\chi^{-1}_0(x) \in W^{k,\infty}_{2\p},
\quad
\nu(x) \in W^{k-1,\infty}_{2\p},
\quad
\hat\s(\br(x)) 
\in W^{k-1,\infty}_{2\p}.
\]
\end{proposition}

\begin{proof}
We know that $0 < C \leq \chi'_0(\t) \leq C'$ for all $\t$, for some constants $C,C'$.
The $k$th derivatives of the inverse diffeomorphism $\chi_0^{-1}(x)$ is a finite sum of terms, each of which is a quotient where the numerator is a polynomial in the derivatives of $\chi_0(\t)$ of order $\leq k$, the denominator is an integer power of $\chi'_0(\t)$, and $\t = \chi_0^{-1}(x)$. Hence also
\[
\chi_0^{-1} \in W^{k,\infty}_{2\p}.
\]
Then the proposition follows by \eqref{x=chi(tau)}, \eqref{coord nu}, \eqref{coord sigma}
and Lemma \ref{lemma:bootstrap}.
\end{proof}

\section[The dynamic boundary condition: Riemann-Hilbert theory]{The dynamic boundary condition:\\ Riemann-Hilbert theory}
\label{sec:RH}

We now derive the dynamic boundary condition for the physical
boundary-value problem and prove part ($f$) of Theorem \ref{thm:all}. 
We recall that the pressure at the free surface in terms of the Lagrangian coordinate $x$ of material points in the reference configuration of the surface membrane is given by \eqref{def P}.
When this is rewritten as a function of $\t$, we find the formula
\begin{equation}  \label{pressure in tau}
P(\t) =
\frac{1}{\Om_0} \Big( \frac{(E_{2,0})'}{\Om_0} \Big)' - \s_0 E_{1,0},
\end{equation}
where $\Om_0 = \Om_0(\t)$ etc., and $'$ is, as usual, the derivative
with respect to $\t$. 
However it is not obvious how to deduce the dynamic boundary condition for hydroelastic waves  directly from the existence of a maximizer of $J_0$. In this section we derive it by interpreting the Euler-Lagrange equation \eqref{weak Euler} as a Riemann--Hilbert problem in the manner of  \cite{Sha-Tol}. 
We begin with a special case of a result in \cite{Toland-1998}, and include a short proof for the sake of completeness.

\begin{lemma}  \label{lemma:Riemann-Hilbert}
Suppose that $f(\t) \in L^\b_{2\p}$, $\b>1$, and $a \in \R$. Then
\[
(i) \ \mC (f w_0') + f (1+\mC w_0') \equiv a \qquad \text{iff}
\qquad (ii) \ \Om_0^2 f \equiv a.
\]
\end{lemma}

\begin{proof} Recall the notation from the proof of Lemma \ref{lemma:inv
mL}.

$(i) \Rightarrow (ii)$. Consider the holomorphic function $U\in
\mH^\beta_\CC$  such that
\[
U^* = f  w_0'+ i (\mC (fw_0') - a), \quad \Im U(0) = - i a.
\]
By $(i)$,
\[
U^* = f \, \overline{W^*},
\]
therefore, multiplying by $W^*$,
\[
U^* W^* = f |W^*|^2 = \Om_0^2 f.
\]
So the holomorphic function $UW\in \mH^\beta_\CC$ is real on the
unit circle. Then on the whole unit disc $UW \equiv b$ for some real
constant $b$. Since at the origin
\[
U(0) W(0) = ([f w_0'] - i a) i =
a + i [f w_0'],
\]
it follows that $b = a$ and $[f w_0']=0$. In particular, on the unit circle
$
a \equiv U^* W^* \equiv \Om_0^2 f.
$

$(ii) \Rightarrow (i)$.
We consider the holomorphic function
\[
V := \frac{a}{W}\in \mH^\infty_\CC\,,
\]
and denote $v := \Re V^*$.
Since $V(0) = a / W(0) = -i a$, on the unit circle
\[
V^* = v + i (-a + \mC v).
\]
On the other hand, $(ii)$ implies that
\[
V^* = \frac{a}{W^*}\, = \frac{a \overline{W^*} }{ |W^*|^2} \,
= f \overline{W^*}
= f w_0' - i f (1 + \mC w_0').
\]
Then
\[
f w_0' = v,
\quad
- f (1 + \mC w_0') = - a + \mC v ,
\]
and $(i)$ follows.
\end{proof}

Let
\[
f := \frac{c^2}{2}\, - g (a_0 + w_0) - P,
\]
where $P$ is defined in \eqref{pressure in tau},
and note that $f \in L^\b_{2\p}$ for all $\b \in (1,\infty)$, by the
regularity results we have proved in the previous section.

By a simple calculation using \eqref{grad I0}, we see that
\begin{equation} \label{dynamics}
\mC (f w_0') + f (1+\mC w_0') \equiv \frac{c^2}{2}
\end{equation}
if and only if
\begin{equation}  \label{claim RH}
 \gr I_0 - g a_0= \mC (w_0' P) + P (1+\mC w_0').
\end{equation}
We now prove that \eqref{claim RH} follows from the two Euler-Lagrange
equations and the regularity results which lead to \eqref{old Q trick}.
After integrating the Euler equation \eqref{weak Euler} by
parts, we get
\begin{align*}
0&=\intp \gr I_0\, h \,d\t + \intp \{ \mC (E_{2,0})' - E_{1,0} \Om_0
\} \, \mL(h')\,d\t
\\
&=\intp \gr I_0\, h \,d\t
- \intp \big\{ \mL^* \big( \mC (E_{2,0})' - E_{1,0} \Om_0 \big)
\big\}'\, h\,d\t
\quad \forall h \in H^{2,\rho}_0,
\end{align*}
where $\mL^*$ is the adjoint operator of $\mL$.
Hence
\[
\gr I_0 - \big\{ \mL^* \big( \mC (E_{2,0})' - E_{1,0} \Om_0 \big) \big\}' \equiv \text{const.},
\]
therefore
\[
\gr I_0 -g a_0 = \big\{ \mL^* \big( \mC (E_{2,0})' -
E_{1,0} \Om_0 \big) \big\}'
\]
because
\[
[\gr I_0] = - g [w_0 \mC w_0'] = g a_0.
\]
Thus, \eqref{claim RH} can be written as
\begin{equation}  \label{claim RH 2}
\{ \mL^* ( \mC T - Q) \}'
- \mC (w_0' P) - P(1+\mC w_0') = 0,
\end{equation}
where, for convenience, we let
\[
T := (E_{2,0})'\in W^{1,\beta}_{2\p}, \quad Q := E_{1,0} \Om_0\in
W^{2,\beta}_{2\p} .
\]
To calculate the left-hand term in \eqref{claim RH 2},
we use the formula
\[
\mL^*u = \frac{w_0' u}{\Om_0^2} - 
\mC \Big( \frac{(1+\mC w_0') u}{\Om_0^2}\Big)
\]
and the equalities
\begin{align*}
\Big( \frac{w_0'}{\Om_0^2} \Big)'
& =
\frac{1+\mC w_0'}{\Om_0^2} \, \Th'_0 -
\frac{w_0'}{\Om_0^2} \, \mC \Th'_0,
\\[5pt]
\Big( \frac{1+\mC w_0'}{\Om_0^2} \Big)' & = - \frac{w_0'}{\Om_0^2}
\, \Th'_0 - \frac{1+\mC w_0'}{\Om_0^2} \, \mC \Th'_0 .
\end{align*}
These follow from the fact that
\[
w_0' = \Om_0 \sin \Th_0, \quad
1 + \mC w_0' = \Om_0 \cos \Th_0, \quad
\Om'_0 / \Om_0 = \mC \Th'_0,
\]
the identity
\[
T'
= \Om_0 \Big( \frac{T}{\Om_0} \Big)'
+ T \, \mC \Th'_0
\]
and the formula
\[
Q' = Q \, \mC \Th'_0 - T \, \Th'_0.
\]
This formula follows from \eqref{old Q trick}, which was obtained by differentiating the Euler equation \eqref{Euler chi}, using the regularity already proved. 
In this way \eqref{claim RH 2}
can be written explicitly, in terms of $\Th'_0, Q, T$ and their
Hilbert transform, as
\begin{align}  \label{monster 20}
& \Big(
\frac{1+\mC w_0'}{\Om_0^2} \, \Th'_0
- \frac{w_0'}{\Om_0^2} \, \mC \Th'_0
\Big)
( \mC T - Q)
 + \mC \Big\{
\Big(
\frac{w_0'}{\Om_0^2} \, \Th'_0
+ \frac{1+\mC w_0'}{\Om_0^2} \, \mC \Th'_0
\Big)
( \mC T - Q)
\Big\}
\notag
\\[4pt]
&  + \ \frac{w_0'}{\Om_0^2} \,
\mC \Big\{
\Om_0 \Big( \frac{T}{\Om_0} \Big)'
+ T \mC \Th'_0 \Big\}
- \ \frac{w_0'}{\Om_0^2} \,
( Q \, \mC \Th'_0 - T \Th'_0 )
\\[4pt]
&  - \ \mC \Big\{
\frac{1+\mC w_0'}{\Om_0^2} \,
\Big(
\mC \Big\{
\Om_0 \Big( \frac{T}{\Om_0} \Big)'
+ T \, \mC \Th'_0
\Big\}
- Q \, \mC \Th'_0 + T \Th'_0 \Big) \Big\}
\notag
\\[4pt]
&  - \ \mC \Big\{
\frac{w_0'}{\Om_0^2} \,
\Big(
\Om_0 \Big( \frac{T}{\Om_0} \Big)' - Q \Th_0'
\Big) \Big\}
- \ \frac{1+\mC w_0'}{\Om_0^2}\,
\Big\{
\Om_0 \Big( \frac{T}{\Om_0} \Big)' - Q \Th_0'
\Big\}
 \, = 0.
\notag
\end{align}
The four terms in \eqref{monster 20} 
involving $u := \Om_0 (T / \Om_0 )'\in L^\beta_{2\p}$ cancel because
\begin{multline}   \label{quotient trick}
\frac{w_0' \, \mC u - (1 + \mC w_0') \, u}{\Om_0^2} \,
- \mC \Big( \frac{w_0' u + (1 + \mC w_0') \, \mC u}{\Om_0^2} \Big)
\\
= \Im \Big( \frac{U^*}{W^*} \Big)
- \mC \, \Re \Big( \frac{U^*}{W^*} \Big)
= 0,
\end{multline}
where $U(z)$ is the holomorphic function of the unit disc such that 
$U^* = u + i \mC u$, and $W$ has been defined in \eqref{def W}.
The eight terms involving $Q$ simply cancel by pairs. 
Now we note that
\begin{equation}  \label{product trick}
\Th_0' T - (\mC \Th_0')(\mC T) + \mC ( T \, \mC \Th_0' )
= \, - \mC ( \Th_0' \, \mC T ),
\end{equation}
because
\[
- \Re \{ (\Th_0' + i \mC \Th_0')(T + i \mC T) \}
= \mC \, \Im \{ (\Th_0' + i \mC \Th_0')(T + i \mC T) \} .
\]
Using \eqref{product trick}, the eight terms involving $T$ cancel because
\[
- \, \frac{w_0' \, \mC \xi - (1 + \mC w_0') \, \xi}{\Om_0^2} \,
+ \mC \Big( \frac{w_0' \xi + (1 + \mC w_0') \, \mC \xi}{\Om_0^2} \Big) = 0,
\]
with $\xi := \Th_0' \, \mC T$, 
for the same reason as in \eqref{quotient trick}.
Hence \eqref{monster 20} holds. This implies \eqref{dynamics} and so,
by Lemma \ref{lemma:Riemann-Hilbert},
\begin{equation}  \label{dyn final}
 1 - \frac{2g}{ c^2}\,(a_0 + w_0) - \frac{2}{c^2}\, P = \frac{1}{\Om_0^2}\,.
\end{equation}
We have showed that, if $(w_0,\chi_0)$ satisfies the Euler equations \eqref{Euler chi} and \eqref{Euler}, then it solves \eqref{dyn final}, proving part ($f$) of Theorem \ref{thm:all}.

Now, the maximum $(w_0,\chi_0)$ for $J_0$ corresponds to the maximum $(w^*_0,\chi_0)$ for $J$, where
\[
w^*_0 := a_0 + w_0
\]
(recall \eqref{conservation of mass} and definition \eqref{def alpha 0}).
Then \eqref{dyn final} writes
\begin{equation}  \label{dyn nonzero mean}
 1 - \frac{2}{c^2}\, ( g w^*_0 + P)  =
\frac{1}{\Om_0^2}\,.
\end{equation}
Note that \eqref{dyn nonzero mean} is the same equation as in
\cite{Toland-steady}, par.\,7.4 
(in \cite{Toland-steady} the pressure was denoted by $-\mF$). 
By Lemma \ref{lemma:bootstrap},
$\Om_0 \in W^{k,\b}_{2\p}$ for all $\b \in (1,\infty)$. So
\[
\frac{1}{\Om_0^2}\, \in W^{k,\b}_{2\p} \quad \forall \b \in (1,\infty),
\]
because $\Om_0 \geq C > 0$ for some $C$.
Hence, by \eqref{dyn nonzero mean},
\[
P = \frac{c^2}{2}\, - g w_0^* - \frac{c^2}{2} \, \frac{1}{\Om_0^2}
\ \in W^{k,\b}_{2\p}
\]
for all $\b \in (1,\infty)$. In particular, in the case when $E$ is
of class $C^2$, the pressure belongs to $W^{2,\b}_{2\p} \subset
W^{1,\infty}_{2\p}$. The same holds for the pressure $P(\br(x))$
(see \eqref{def P}) as a function of the Lagrangian coordinate $x$.

\begin{lemma}  \label{lemma:maths => phys}
The solution $(w_0,\chi_0)$ of the variational problem described in Theorem \ref{thm:all} gives a solution of the physical problem \eqref{classical problem}, where the free boundary $\mS$ and the material deformation of the membrane $\br(x)$ are
\[
\mS = \mS(w_0) := \{ \rho(w_0)(\t) : \t \in \R \},
\quad \br(x) = \rho(w_0) (\chi_0^{-1}(x)).
\]
\end{lemma}

\begin{proof}
It was proved in \cite{Sha-Tol} that, given the curve $\mS(w)$
parametrized by \eqref{curve parametrized}, the solution $\widetilde
\psi$ of the problem
\begin{align*}
\D \widetilde \psi = 0 & \quad \text{below } \mS(w), \\
\widetilde \psi = 0 & \quad \text{on } \mS(w), \\
\gr \widetilde \psi \to (0,1) & \quad \text{as } Y \to -\infty
\end{align*}
satisfies
\[
|\gr \widetilde \psi |^2 = \frac{1}{\Om(w)^2} 
\quad \text{on }\mS(w).
\]
Now, we consider the curve $\mS = \mS(w_0)$, where $(w_0,\chi_0)$ is
the solution of the variational problem  in Theorem \ref{thm:all}.
The solution $\psi$ of problem (\ref{classical problem}a,b,c) is
$\psi = c \widetilde \psi$, therefore
\[
|\gr \psi |^2 = \frac{c^2}{\Om(w)^2} \quad \text{on }\mS(w).
\]
By construction, $w^*_0(\t)$ is the elevation of the point $\rho(w)(\t) = \br(x)$ of the membrane, that is the vertical coordinate of the deformed point $\br(x)$ with respect to the rest frame, and $P(\t)$ is the pressure at the point $\rho(w)(\t) = \br(x)$.
Then \eqref{dyn nonzero mean} is exactly \eqref{psi dynamics}.

Note that the constraint \eqref{r constrain} is satisfied by $(w_0,\chi_0)$ by construction.
\end{proof}

\begin{small}

\end{small}

\bigskip

\noindent
Department of Mathematical Sciences,\\
University of Bath,\\
Bath BA2 7AY, U.K. 

\smallskip
 
\noindent
\emph{E-mail:} \texttt{P.Baldi@bath.ac.uk,\, jft@maths.bath.ac.uk}

\end{document}